\documentclass{article}
\usepackage[utf8]{inputenc}
\usepackage[T1]{fontenc}
\usepackage{amsmath}
\usepackage{amsfonts}
\usepackage{amssymb}
\usepackage{amsthm}
\usepackage{graphicx}
\usepackage{pgf,tikz}
\usepackage{hyperref}
\usepackage{float}
\usepackage{geometry}
\usepackage[labelsep=space]{caption}
\usepackage{subcaption}
\usepackage{enumitem}
\usepackage{multicol}

\hypersetup{
	colorlinks=true,
	linkcolor=blue,
	filecolor=blue,      
	urlcolor=blue,
	citecolor=blue
      }

\newtheorem{definition}{Definition}
\newtheorem{thm}{Theorem}
\newtheorem{prop}{Proposition}
\newtheorem{rem}{Remark}
\newtheorem{lem}{Lemma}
\newtheorem{cor}{Corollary}

\definecolor{myGreen}{rgb}{0.0, 0.0, 0.0}
\newcommand{\D}{{\color{myGreen} \bar D}}
\newcommand{\green}[1]{\textcolor{myGreen}{#1}}
\newcommand{\mix}{\text{\normalfont mix}}
\newcommand{\IT}{\text{\normalfont i}}
\newcommand{\FTD}{\text{\normalfont f}}
\newcommand{\tv}[1]{\lVert #1 \rVert_{\text{\normalfont TV}}}

\begin{document}
\author{Arnaud Guillin\thanks{arnaud.guillin@uca.fr} \and 
	Leo Hahn\thanks{leo.hahn@unine.ch}\;\thanks{Corresponding author} \and 
	Manon Michel\thanks{manon.michel@uca.fr}}
\date{\today}

\title{Long-time analysis of a pair of on-lattice and continuous run-and-tumble particles with jamming interactions}

\maketitle

\begin{center}
	\small
	Laboratoire de Mathématiques Blaise Pascal UMR 6620, CNRS, Université Clermont Auvergne, Aubière, France.
\end{center}

\begin{abstract}
Run-and-Tumble Particles (RTPs) are a key model of active matter. They are characterized by alternating phases of linear travel and random direction reshuffling. By this dynamic behavior, they break time reversibility and energy conservation at the microscopic level. It leads to complex out-of-equilibrium phenomena such as collective motion, pattern formation, and motility-induced phase separation (MIPS). In this work, we study two fundamental dynamical models of a pair of RTPs with jamming interactions and provide a rigorous link between their discrete- and continuous-space descriptions. We demonstrate that as the lattice spacing vanishes, the discrete models converge to a continuous RTP model on the torus, described by a Piecewise Deterministic Markov Process (PDMP). This establishes that the invariant measures of the discrete models converge to that of the continuous model, which reveals finite mass at jamming configurations and exponential decay away from them. This indicates effective attraction, which is consistent with MIPS. Furthermore, we quantitatively explore the convergence towards the invariant measure. Such convergence study is critical for understanding and characterizing how MIPS emerges over time. Because RTP systems are non-reversible, usual methods may fail or are limited to qualitative results. Instead, we adopt a coupling approach to obtain more accurate, non-asymptotic bounds on mixing times. The findings thus provide deeper theoretical insights into the mixing times of these RTP systems, revealing the presence of both persistent and diffusive regimes.
\end{abstract}

\bigskip

\noindent{\bf Keywords:} run-and-tumble particles, piecewise
deterministic Markov processes, jamming, mixing time.
















	

\section[Introduction]{Introduction, models and main results}

%
%
%

\subsection{Introduction}

Run-and-tumble particles (RTPs), which randomly switch between phases
of linear motion (runs) through localized chaotic motion (tumbles)
resulting in reorientation, are a prime example of active
matter~\cite{ramaswamy10}. Their self-propulsion breaks time
reversibility and energy conservation at the microscopic scale, two
fundamental properties of equilibrium systems. This leads to rich
out-of-equilibrium behavior such as collective motion~\cite
{vicsek12}, pattern formation~\cite{budrene95} and motility induced
phase separation (MIPS)~\cite{cates15}. In particular, accumulation at
obstacles and MIPS both result in spatially inhomogeneous particle
distributions, that deviate from the homogeneous passive particle
ones. Thus they cannot a priori be framed into the equilibrium
description based on an already known stationary distribution and
their dynamics be summarized into a reversible random
walk. Coarse-grained analytical approaches allowed to retrieve some
observed mesoscopic behaviors, but are unfitted to trace back an
observed behavior to a particular microscopic origin. More generally,
there is still no systematic method for deriving directly from the
microscopic details the exact form of the stationary
distribution.

Some recent results were obtained for particular on-lattice models,
consisting of two RTPs interacting through
jamming~\cite{slowman16,slowman17}\textcolor{black}{, which translates
  into blocking configurations due to hardcore interactions and the
  persistence of RTP motion}. The continuous-space invariant measure,
formally obtained by making the lattice spacing go to zero, displays
increased mass at jamming configurations. This increased mass takes
the form of Diracs and, in the case of nonzero tumble
duration~\cite{slowman17}, exponentials, indicating effective
attraction in line with MIPS. The same discretization scheme is
revisited for models with collision-induced recoil~\cite
{metson22,metson23}, leading to either effective attraction or
repulsion depending on model parameters. {\color{black} In parallel,
  direct continuous-space models have been introduced to study the
  stationary bound state of $N$ interacting RTPs, though these either
  include thermal noise \cite{das20} or simply neglect jamming effects
  \cite{basu20,ledoussal21,touzo24}. 
  A key challenge in
  continuous space is that jamming interactions translate into
  boundary conditions on the state space, which prior approaches
  usually struggled to handle rigorously.  To address this, a
  framework based on piecewise deterministic Markov processes
  (PDMPs)~\cite{davis93} was recently proposed~\cite{hahn23}. This
  method allows direct modeling in continuous space, simplifies
  computations, and—crucially—enables a more tractable treatment of
  boundary conditions, even in the absence of thermal noise.  A clear
  understanding of these boundary effects has now been shown to be
  essential for computing invariant measures in such
  systems~\cite{hahn23}. Using the PDMP framework, the explicit
  invariant measure of two interacting RTPs with general tumbling and
  jamming behavior was computed, revealing the existence of two
  universality classes—distinguished by whether probability flow is
  preserved in a detailed or global sense at the boundaries.}  While
the invariant measures obtained with the PDMP approach in
\cite{hahn23} recover the continuous limit of the discrete processes
in \cite{slowman16,slowman17}, the proof of such convergence, as the
lattice spacing decreases, is still missing.

The computation of the invariant measures naturally leads to the
question of the speed of convergence towards it. In the context of
RTPs, this is linked to the time it takes for MIPS to emerge from an
arbitrary configuration. Indeed, while the invariant measure of Markov
processes determines their behavior at time scales of the order of
their mixing time, it becomes irrelevant at smaller time
scales.
{\color{black} While the related question of the first collision time of 2 RTPs~\cite{ledoussal19}, a form of first passage time~\cite{angelani14,malakar18,gueneau24}, has been studied, the speed of convergence toward the invariant measure is more rarely investigated.} The most commonly used approach is spectral
decomposition: the distance to the invariant measure decays
exponentially, with an asymptotic rate determined by the second
eigenvalue of the generator. The full spectral decomposition of two
on-lattice RTPs on the 1D torus with jamming is computed
in~\cite{mallmin19}, revealing eigenvalue crossings and a non-analytic
spectral gap. In continuous space with added thermal
noise~\cite{das20}, the generator is a differential operator so the
eigenvalue problem takes the form of an ODE.\@ The full spectrum is
again computed and the spectral gap is found to reduce to the spectral
gap of the velocity dynamics as the length of the 1D torus goes to
zero. This occurs because the time it takes to explore position space
goes to zero while the time to explore velocity space remains
constant. In~\cite{angelani19}, the eigenvalues are obtained as the
poles of the Laplace transform in time of the position distribution
and the same reduction phenomenon is observed.  Unfortunately, when
the generator is an unbounded operator, the set of eigenvalues can be
a proper subset of the spectrum, complicating the computation of the
spectral gap. Furthermore, because RTP systems are non-reversible,
there is an uncontrolled prefactor to the exponential decay and the
convergence results are only true in the asymptotic regime, leaving
open the question in the non-asymptotic case.

 In this work, we address these different convergence questions. The
 first result of this paper is that the continuous-space limit, where
 the lattice spacing tends to zero, of the Markov jump processes
 in~\cite{slowman16,slowman17} is a variant of the PDMP
 in~\cite{hahn23}, rigorously linking the two models. The convergence
 is made quantitative by relying on a coupling strategy rather than
 the usual generator approach~\cite{ethier86}. A uniformity result on
 the mixing times of the discrete-space processes is then used to
 prove that the on-lattice invariant measures converge to the
 continuous invariant measure for the Wasserstein distance, thus
 giving further theoretical underpinning to~\cite{slowman16,slowman17}.
 Secondly, regarding the long-time behavior, we adopt
 the coupling approach~\cite{fontbona12} to obtain sharp
 non-asymptotic bounds.  An order preservation argument, in the
 general spirit of~\cite{lund96}, reduces the coupling approach to
 hitting time computations and provides quantitative upper bounds on
 the mixing time. Matching lower bounds are established by identifying
 obstacles to mixing, which are then made quantitative by
 concentration inequalities. This shows that the bounds are optimal
 and fully capture the dependence of the mixing time on model
 parameters, incidentally revealing the existence of a persistent and
 diffusive regime similar to~\cite{das20,angelani19}.

In the remainder of this section, the discrete models from~\cite
{slowman16,slowman17}, as well as their continuous-space limit, are introduced
and the main results are presented. The section ends with a plan of the paper.

\subsection{Models and main results}
\label{sec:construction_of_the_stochastic_processes}

\subsubsection{Discrete process}

All processes considered in this paper aim to describe the collective
ballistic motion of a pairs of RTPs, that jam when they collide. We
define the Markov jump process introduced in \cite{slowman16,slowman17}.
The goal is to show that this discrete model converges
in law to particular instances of the general continuous-space stochastic
process considered in \cite{hahn23}.

The discrete process of interest corresponds to the setting of two
RTPs on a periodic one-dimensional lattice of $L+1$ sites, as
illustrated in Fig.~\ref{fig:rtp_illust}. Each particle is endowed
with a velocity and an independent Poisson clock with rate
$\gamma$. When their clock rings, the particles jump to the
neighboring position in the direction of their velocity except if that
position is occupied by the other particle. The particles thus display
the persistent motion characteristic of RTPs as well as jamming
interactions.

\begin{figure}[H]
	\centering
	\includegraphics[width=0.8\linewidth]{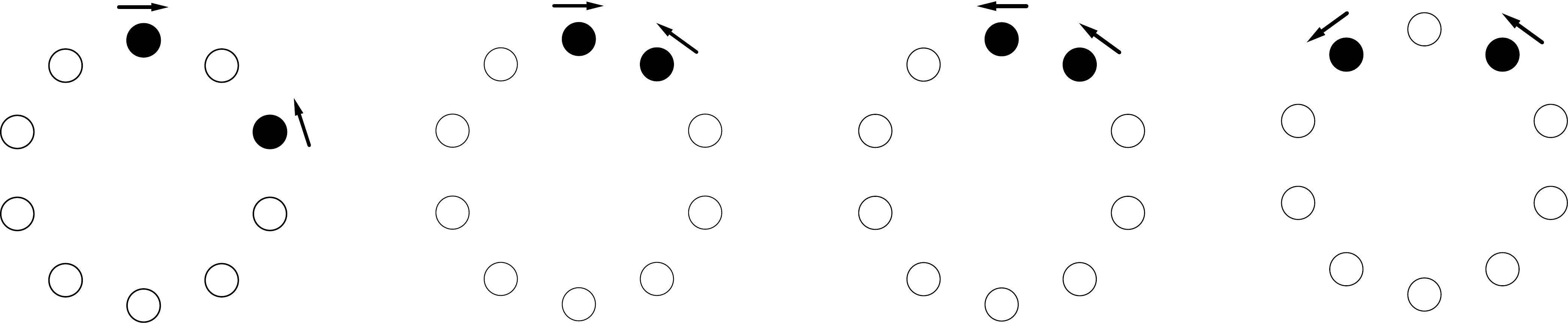}
	\caption{Sequence of random position jumps and velocity flips}
        \label{fig:rtp_illust}
\end{figure}

In \cite{slowman16}, the tumbles are approximated by instantaneous
reorientations and the velocities are independent Markov jump
processes taking the values $1$ and $-1$ and following the transition
rates of figure
\ref{fig:single_speed_transition_rates_slowman_2016}. In
\cite{slowman17}, the tumbles have a finite nonzero duration so that
the velocities can take the additional value $0$ (tumbling and not
moving) and follow the transition rates of
figure~\ref{fig:single_speed_transition_rates_slowman_2017}.

\begin{figure}[H]
	\centering
	\begin{subfigure}{0.3\textwidth}
		\centering
		\begin{tikzpicture}[->, node distance={2cm}, thick, main/.style = {draw, circle, minimum size=1cm}] 
		\node[main] (I+) {$+1$};
		\node[main] (I-) [right of=I+] {$-1$};
		
		\path[every node/.style={font=\sffamily\small}]
		(I+) edge[bend left] node [above] {$\omega$} (I-)
		(I-) edge[bend left] node [below] {$\omega$} (I+)
		;
		\end{tikzpicture}
		\caption{Instantaneous tumble}
		\label{fig:single_speed_transition_rates_slowman_2016}
	\end{subfigure}
	\begin{subfigure}{0.3\textwidth}
		\centering
		\begin{tikzpicture}[->, node distance={2cm}, thick, main/.style = {draw, circle, minimum size=1cm}] 
		\node[main] (F+) [right of=I-] {$+1$};
		\node[main] (F0) [right of=F+] {$0$};
		\node[main] (F-) [right of=F0] {$-1$};
		\path[every node/.style={font=\sffamily\small}]
		(F+) edge[bend left] node [above] {$\alpha$} (F0)
		(F0) edge[bend left] node [below] {$\beta/2$} (F+)
		(F0) edge[bend right] node [below] {$\beta/2$} (F-)
		(F-) edge[bend right] node [above] {$\alpha$} (F0)
		;
		\end{tikzpicture}
		\caption{Finite tumble}
		\label{fig:single_speed_transition_rates_slowman_2017}
	\end{subfigure}
	\caption{Transition rates of the velocity of a single particle}
	\label{fig:single_speed_transition_rates_slowman_2017_and_2017}
\end{figure}
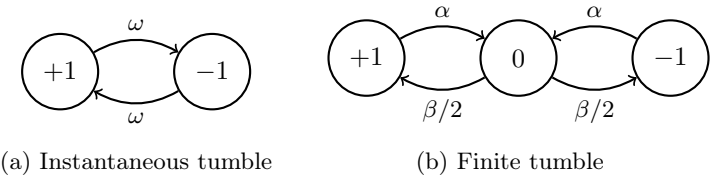

Because the process is translation invariant, it is enough to consider
the inter-particle separation $y^L \in \{ 1, \ldots, L \}$, defined as the difference from the first to
the second RTP position modulo $L+1$, and the velocities $\sigma^1$
and $\sigma^2$ of the two particles. The inter-particle separation $y^L$ stays
constant until the clock of, say, the first particle rings at time $t$, and, 
\begin{itemize}
\item if $y^L(t-) - \sigma^1(t-) \in \{ 1, \ldots, L\}$ then $y^L$ jumps to $y^L(t-) - \sigma^1(t-)$ (travel to free site),
\item if $y^L(t-) - \sigma^1(t-) \notin \{ 1, \ldots, L\}$ then nothing happens (jamming).
\end{itemize}
If the clock of the second particle rings, $y^L$ is updated in the same manner once $ - \sigma^1(t-)$ is replaced by $+ \sigma^2(t-)$.
        
Therefore the process
$Y^L(t) = (y^L(t), (\sigma^1(t), \sigma^2(t)))$ is a Markov jump process
and its construction is laid out in the following definition.

\begin{definition}[Discrete process] \label{def:discrete_process}

  Let $L$ be a positive integer,
  $y^L_0 \in \{ 1, \ldots, L \}$  and
  $\sigma^1_0, \sigma^2_0 \in \{-1, 1\}$ (resp.
  $\sigma^1_0, \sigma^2_0 \in \{-1, 0, 1\}$) . Let $\sigma^1$ and $\sigma^2$ be two
  independent Markov jump processes with the transition rates of
  figure \ref{fig:single_speed_transition_rates_slowman_2016}
  (resp. \ref{fig:single_speed_transition_rates_slowman_2017}) and respective initial states \( \sigma^1_0 \) and \( \sigma^2_0 \). Further let
  $N_1$ and $N_2$ be two independent Poisson clocks with rate $\gamma$
  independent of the $\sigma^i$. Recursively define $y^L(t)$ by $y^L(0) = y^L_0$ and,
\begin{itemize}
\item when $N_1$ rings at time $t$, $y^L(t)=\max(1,\min(L, y^L(t-) - \sigma^1(t-)))$.
\item when $N_2$ rings at time $t$, $y^L(t)=\max(1,\min(L, y^L(t-) + \sigma^2(t-)))$
\end{itemize} 
\end{definition}

To emphasize which $\sigma^i$, $N_i$, $y_0^L$ and $L$ were used, it is sometimes
useful to denote the process constructed in definition
\ref{def:discrete_process} by $\mathcal Y(\sigma^1, \sigma^2, N_1, N_2, y^L_0,
 L)$. The process $Y^L$ is referred to as {discrete instantaneous tumble
 process} (DITP) when the velocities follow figure
\ref{fig:single_speed_transition_rates_slowman_2016} and as {discrete finite
 tumble process} (DFTP) when the velocities follow figure
\ref{fig:single_speed_transition_rates_slowman_2017}.	

\subsubsection{Continuous process}\label{sec_construction}

We now present the continuous process, that we show is the
continuous-space limit of the discrete process under the appropriate
rescaling. \textcolor{myGreen}{This process can be described in two
  different manners, either an ad-hoc trajectory one or a more general
  generator one, as done in \cite{hahn23}. We detail now the more concise
  trajectory description and refer to Section~\ref{secdavis} for a
  complete PDMP construction that follows the original Davis construction
  \cite{davis93}. We then show the equivalency of both descriptions and later
  use one or the other approach depending on their respective advantage.}

Consider two point particles, each endowed with a velocity following a
Markov jump process with the transition rates of figure
\ref{fig:single_speed_transition_rates_slowman_2016}
(resp. \ref{fig:single_speed_transition_rates_slowman_2017}), according to which they
move around a 1D torus of length $\ell > 0$.
The particles jam when they collide and therefore they can never pass
through each other. Similarly to $y^L$ in the discrete case, we
consider the inter-particle separation $x \in [0, \ell]$, identifying
with the difference from the first to the second particle position
modulo $\ell$. Denote $\sigma^1$ and $\sigma^2$ the velocities of the
particles and $[t_1, t_2]$ a time interval on which the velocities are
constant.
\begin{itemize}
\item If $\sigma^2(t_1) - \sigma^1(t_1) > 0$ then $x$ increases at
  constant velocity $\sigma^2(t_1) - \sigma^1(t_1)$ until it reaches
  $\ell$ (jamming configuration) hence
$$
x(t) = \min \left[\ell, x(t_1) + (\sigma^2(t_1) - \sigma^1(t_1))(t -
  t_1) \right] \text{ for } t \in [t_1, t_2].
$$
\item If $\sigma^2(t_1) - \sigma^1(t_1) < 0$ then $x$ decreases at
  constant velocity $\sigma^2(t_1) - \sigma^1(t_1)$ until it reaches $0$
  (jamming configuration) hence
$$
x(t) = \max \left[0, x(t_1) + (\sigma^2(t_1) - \sigma^1(t_1))(t - t_1)\right]\text{ for } t \in [t_1, t_2].
$$
\item If $\sigma^2(t_1) - \sigma^1(t_1) = 0$ then $x$ is constant on
  $[t_1, t_2]$.
\end{itemize}

This leads to the following recursive construction for the continuous
process (see figure~\ref{fig:continuous_instanteneous_tumble_realization} for a realization).

\begin{definition} \label{def:continuous_process_min_max_construction}(Continuous
  process) Let $\ell>0$,  $x_0 \in [0, \ell]$ and
  $\sigma^1_0, \sigma^2_0 \in \{-1, 1\}$ (resp.
  $\sigma^1_0, \sigma^2_0 \in \{-1, 0, 1\}$) be given. Consider two
  independent Markov jump processes $\sigma^1$ and $\sigma^2$ with the
  transition rates of figure
  \ref{fig:single_speed_transition_rates_slowman_2016}
  (resp. \ref{fig:single_speed_transition_rates_slowman_2017}) and
  respective initial states $\sigma^1_0$ and $\sigma^2_0$.

  Using the jump times $(T_i)_{i \ge 0}$ of the couple
  $(\sigma^1, \sigma^2)$, recursively define $x(t)$ by $x(0) = x_0$
  and
\begin{equation*}
x(t) = \max \left[0,  \min \left[\ell, x(T_i) + (\sigma^2(T_i) - \sigma^1(T_i)) (t - T_i) \right] \right] \text{ for } t \in [T_i, T_{i+1}[.
\end{equation*}
The process $X(t) = (x(t), (\sigma^1(t), \sigma^2(t)))$ is referred to as {continuous process}.
\end{definition}

\begin{figure}[H]
	\centering
	\includegraphics[width=0.8\textwidth]{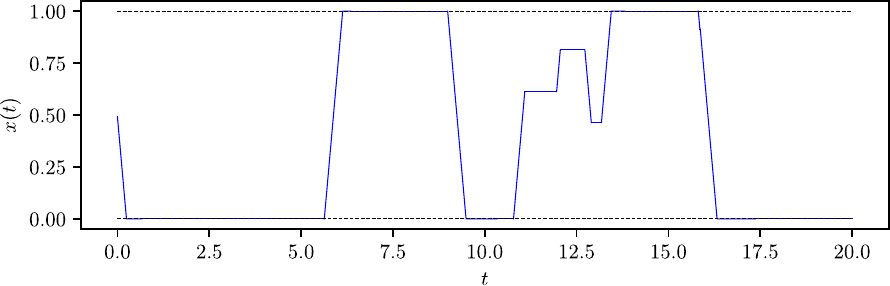}
	\caption{Realization of $x$ for the CITP with $\omega = 0.5$ and $\ell = 1$}\label{fig:continuous_instanteneous_tumble_realization}
\end{figure}

In order to emphasize which $\sigma^i$, $x_0$ and $\ell$ were used, the
process constructed in the previous definition is sometimes denoted by
$\mathcal X(\sigma^1, \sigma^2, x_0, \ell)$. When the $\sigma^i$ follow the
rates of figure
\ref{fig:single_speed_transition_rates_slowman_2016}, $X$ is called the
 continuous instantaneous tumble process (CITP) and when the $\sigma^i$
 follow the rates of figure
\ref{fig:single_speed_transition_rates_slowman_2017}, $X$ is called the
 continuous finite tumble process (CFTP).

Starting from definition
\ref{def:continuous_process_min_max_construction}, the proof of the
strong Markov property and the characterization of the generator 
requires careful work. Instead, as done in \cite{hahn23}, we construct the process
using the formalism of piecewise deterministic Markov processes
(PDMPs) developed in~\cite{davis93} so that the strong Markov property
and the characterization of the generator are immediate. As done in
\cite{bierkens23}, a bijection $\iota$ mapping the state space of the
PDMP to a more convenient state space is used to come back to
definition \ref{def:continuous_process_min_max_construction}. It is
the subject of Section \ref{secdavis}.

\subsubsection{Main results}

The following theorem is a summary of the findings obtained in this paper.

\begin{thm}\label{thm:summary}
Let $\gamma_L=(L-1)/\ell$. 
\begin{enumerate}[label={\normalfont(\arabic*)}]
\item The discrete in space processes DITP (resp. DFTP) converge, with quantitative control,  when $L$ tends to infinity, in Skorohod topology towards the continuous-space piecewise deterministic Markov process CITP (resp. CFTP).
\item The processes DITP and DFTP are exponentially ergodic with invariant measures converging as $L$ tends to infinity to the unique invariant measures of the CITP and CFTP.
\item The invariant measures of CITP and CFTP have a positive mass at $\{0\}\times \{-1,0,1\}$ and  $\{\ell\}\times \{-1,0,1\}$.
\item Let us define for a Markov process with semigroup $P_t$ and invariant measure $\pi$ its mixing time as
$t_\mix(\epsilon)=\inf\left\{t\ge0\,;\, sup_\mu\|\mu P_t-\pi\|_{TV}\le \epsilon\right\}$.
\begin{enumerate}
\item[\normalfont (a)] The mixing time $t_\mix(\epsilon)$ of CITP is of the order $C'(\epsilon)\frac 1\omega(1+\omega^2l^2)$ for some $C'(\epsilon)$.
\item[\normalfont (b)] The mixing time $t_\mix(\epsilon)$ of CFTP is of the order $C'(\epsilon)(\frac1\alpha+\frac1\beta)(1+\alpha^2l^2)$ for some $C'(\epsilon)$.
\end{enumerate}
\end{enumerate}
\end{thm}

The precise statements of points (1)--(3) are given in Theorem
\ref{thm:quantitative_discrete_continuous_convergence}, Corollary
\ref{cor:limit_interchange} and in the appendix. It enables to justify
the analysis of \cite{slowman16,slowman17} who have calculated the
invariant measure of the DITP and DFTP and have passed to the limit in
order to assert that this processes exhibit a MIPS phenomenon,
i.e. only their intrinsic movement allows for agglomerate of particles
asserted by a positive mass at $0$ and $\ell$.  This result not only
ensures the validity of their analysis but also provides a direct way
to build processes (PDMPs) at the continuous level. Note that the
explicit invariant measures obtained in the Appendix is a particular
case of the general study of \cite{hahn23} where universality classes
(for the invariant measures) were derived for general pairs of RTPs
depending on the behavior at the jamming boundary.

{\color{black} The coupling strategy used here to derive the discrete
  to continuous limit may be extended to tumble mechanisms beyond the
  DITP~\cite{slowman16} and DFTP~\cite{slowman17}, although doing so
  would require certain adaptations. A more general framework, capable
  of encompassing arbitrary tumbling mechanisms as considered
  in~\cite{hahn23}, could be more naturally formulated using a
  generator-based approach (see
  Remark~\ref{rem:no_proof_of_discrete_continuous_convergence_in_general}). However,
  this comes at the expense of losing quantitative control over the
  convergence bounds. For this reason, we focus in the following on a
  coupling-based analysis applied specifically to the CITP and CFTP.}

The mixing times were not studied before, despite the importance of the understanding of the speed of convergence to stationarity. The
precise statements of the results are given in Theorem
\ref{thm:mixing_time}. Most interestingly, it reveals two regimes controlled by the parameter \( \omega \ell \) for the CITP
\begin{itemize}
\item the {\it persistent} regime $\omega \ell \ll 1$ where the mean time between two stochastic
  velocity changes is much greater than the time needed to reach a
  jammed configuration,
\item the {\it diffusive} regime $\omega \ell \gg 1$ where the mean time between two stochastic
  velocity changes is much lower than the time needed to reach a jammed
  configuration.
\end{itemize}
This allows for a heuristic interpretation of the mixing time in terms
of slow observables. In the persistent regime, the slowest observables
are the velocities $\sigma^i$ so the mixing time is of order $1/\omega$. In the diffusive regime, the position
$x$ is the slowest observable and its dynamic is approximately
Brownian leading to the mixing time $\omega \ell^2$ with the quadratic dependence in
$\ell$ characteristic of diffusions.

For the CFTP, the mixing time is the product of
\( \frac1\alpha + \frac 1\beta \), the average time needed to see a velocity
flip of each kind, and \( 1 + \alpha^2 \ell^2 \) which accounts for
the possibly diffusive behavior. Notice that the mixing times in
theorem~\ref{thm:summary}~(4) are coherent with the fact that, when
\( \beta \to +\infty \), the CFTP becomes the CITP with
\( \omega = \alpha/2 \).

\subsubsection{Plan of the paper}

In the following, the stochastic processes CITP and CFTP are
constructed in section~\ref{secdavis} considering the PDMP formulation
of Davis \cite {davis93}, with a particular emphasis on their
generator. The continuous PDMP description proves helpful to derive
the explicit form of the invariant measure, as was generally done in
\cite{hahn23}, but also for mixing time evaluation,
\textcolor{myGreen}{as can be seen in
  Lemma~\ref{lem_maximum_hitting_time_of_0}}. In section
\ref{sec:from_discrete_to_continuous_space}, it is shown that, under
the right rescaling, the discrete process converges in law to the
continuous process with respect to the Skorokhod metric. It is also
proven that the on-lattice invariant measures converge to the
continuous invariant measure with respect to the Wasserstein
distance. Proofs are based on coupling between the discrete state
process and the continuous one. In section
\ref{sec:mixing_time_of_the_continuous_process}, the scaling of the
mixing time of the continuous process as a function of model
parameters is determined. The arguments leading to such precise
controls on the mixing time relies on sharp controls of expectation of
particular hitting times of the continuous-space processes, as well as
concentration of additive functionals of the velocity
processes. Finally, for the sake of completeness, the explicit formula
for the invariant measure of the continuous processes is recalled in
the appendix, directly for CITP and CFTP, using a similar but
alternative approach to that in~\cite{hahn23}.

\section{PDMP formulation of the continuous-space
processes}\label{secdavis}

Piecewise deterministic Markov processes are characterized by a
deterministic dynamic interspersed with random jumps. Their appeal
lies in the wide range of processes that fall into this category and
their straightforward construction. The complete characterization of
their generator and its domain is known, even if sometimes in an
intrinsic way. This greatly facilitates the search for their invariant
measures. \textcolor{myGreen}{In \cite{hahn23}, the PDMP formalism enables the explicit
derivation of the invariant measure and its
universality classes for two RTP in a general setting. The use of
generator is also central for the Foster-Lyapunov function method of
Meyn-Tweedie \cite{meyn09} to get rates of convergence to equilibrium.}
Furthermore, the formalism conveniently includes a way of explicitly
computing the expected value and the Laplace transform of hitting
times by solving systems of differential equations.

\textcolor{myGreen}{\subsection{PDMP construction}}

\textcolor{myGreen}{The goal of this section is to construct following
  \cite[Section 24]{davis93} the generator of the PDMP characterizing
  the process CITP (resp. CFTP) defined in
  subsection~\ref{sec_construction}. It slightly differs from the
  construction in \cite{hahn23} and can be skipped on a first read. In
  the latter, the considered PDMP encodes all the system symmetries
  (particle indistinguishability, periodicity, space homogeneity) and
  thus corresponds to the evolution of the interdistance between the
  two particles. Here,} we are interested in the limit of the discrete
process $Y^L$ studied in \cite{slowman16,slowman17} and the mixing
behavior for the particular case of the instantaneous and finite
tumbles. Thus, here, we \textcolor{myGreen}{construct the PDMP}
describing the evolution of $X_D(t) = (x(t),\nu(t))$ with \textcolor{myGreen}{
  \begin{itemize}
    \item $x(t)\in [0,\ell]$ the inter-particle separation, as
      previously defined,
    \item $\nu(t)\in K$ a variable ruling the deterministic dynamics,
      that encodes both the value of the velocities $\sigma^i$ and
      whether or not the particles are jammed;
      Indeed $K$ identifies with $\Sigma \cup \partial K$ where,
      \begin{itemize}
      \item $\Sigma = \{-1, 1\}^2$ for the CITP (resp. $\Sigma = \{-1, 0, 1\}^2$ for the CFTP); $\nu\in \Sigma$ corresponds to an unjammed state where the RTP have for velocities $\sigma=(\sigma^1,\sigma^2)$,
      \item and $\partial K = \left\{ *^\ell_\sigma : \sigma \in \Sigma_+\right\} \cup \left\{ *^0_\sigma : \sigma \in \Sigma_-\right\}$ with $\Sigma_+ = \{ (\sigma^1, \sigma^2) \in \Sigma : \sigma^2 - \sigma^1 \ge 0\}$  and $\Sigma_-~=~\{ (\sigma^1, \sigma^2)~ \in~ \Sigma : \sigma^2 - \sigma^1 \le 0\}$; $\nu\in\partial K$ corresponding to a jammed state where the RTP have for velocities $\sigma=(\sigma^1,\sigma^2)$,
        \end{itemize}        
\item and the subscript $D$ stands for Davis.
\end{itemize}}
\begin{figure}
	\centering
	\begin{subfigure}[t]{0.3\textwidth}
		\centering
		\includegraphics[scale=0.12]{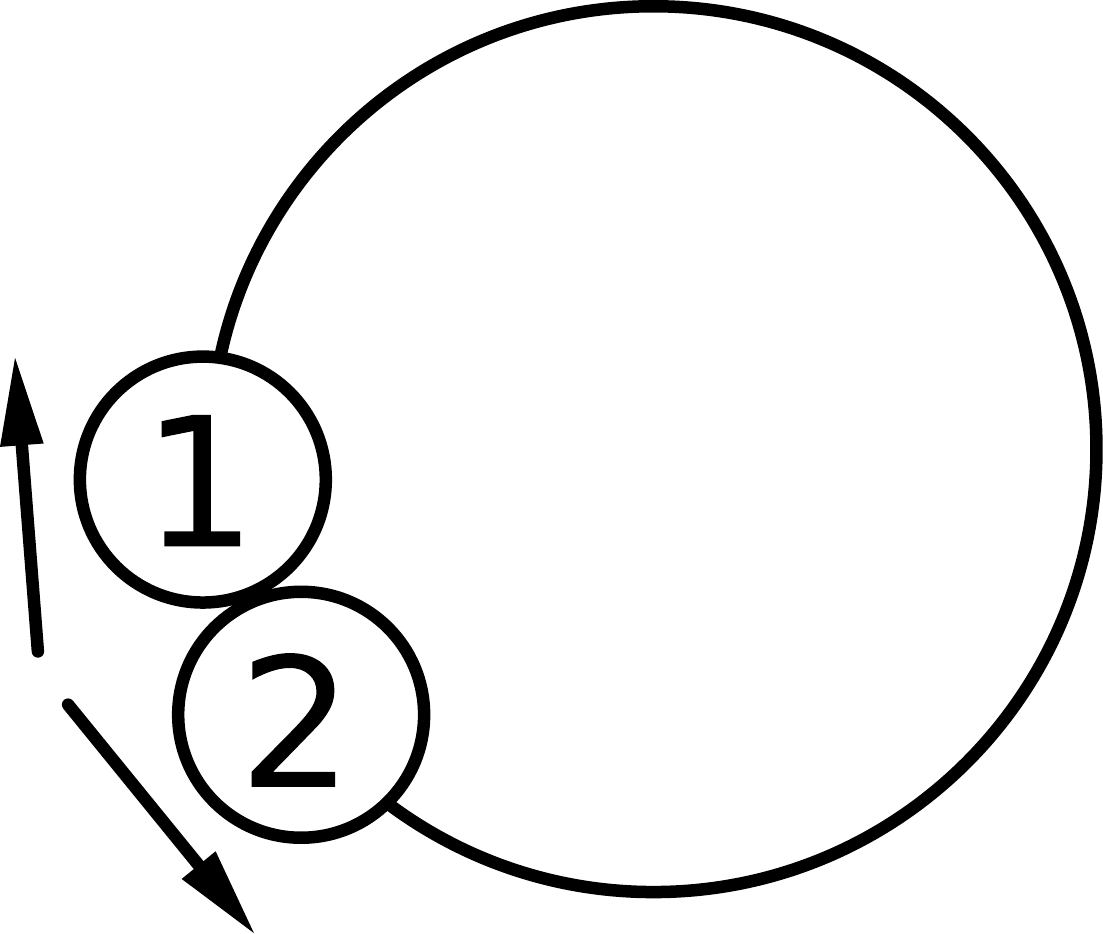}
		\caption{$x = 0$ and $\nu = {(-1, 1)}\in \Sigma$}
		\label{fig:explaining_the_state_space_bulk_contact}
	\end{subfigure}
	\begin{subfigure}[t]{0.3\textwidth}
		\centering
		\includegraphics[scale=0.12]{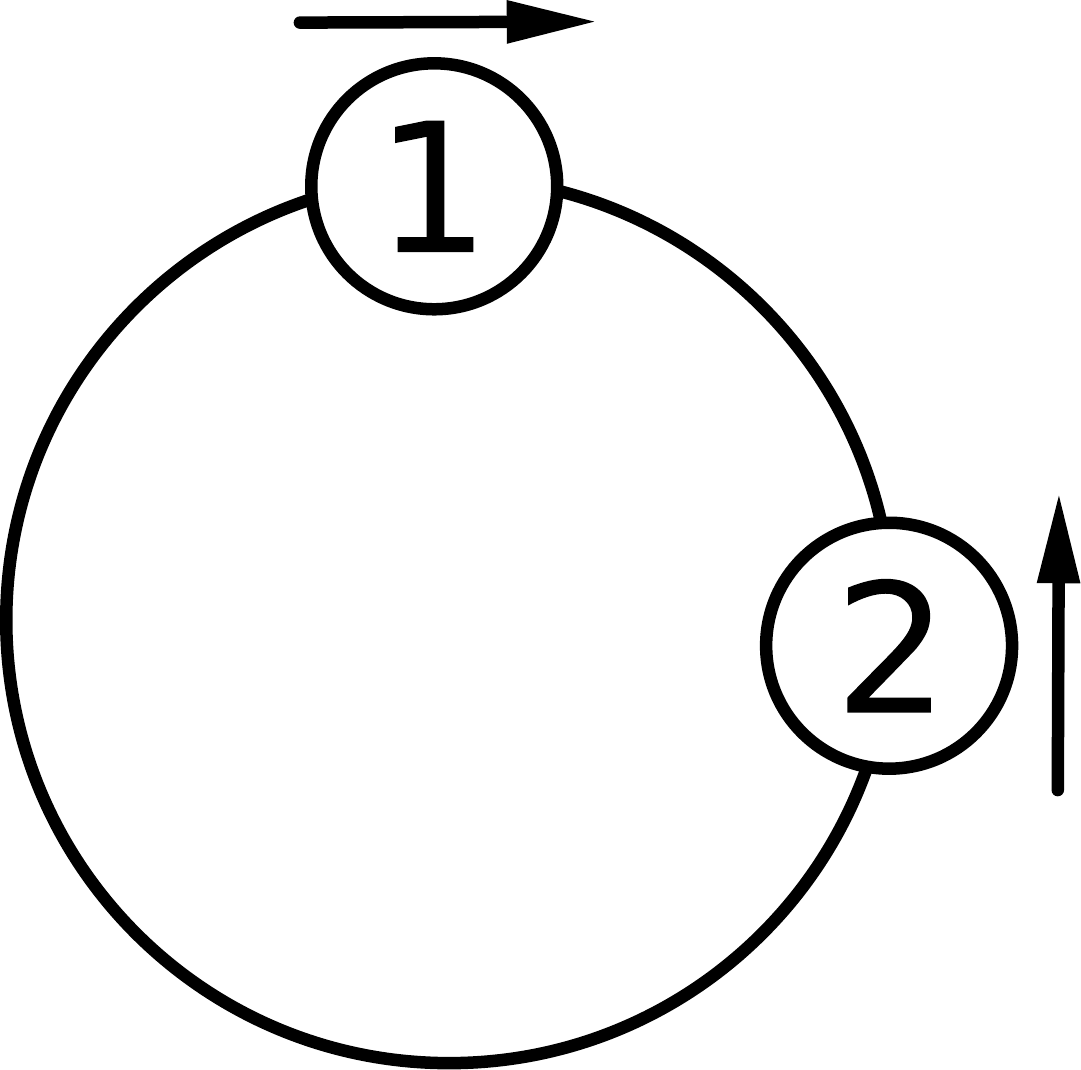}
		\caption{$0 < x < \ell$ and $\nu = (-1, 1)\in\Sigma$}
		\label{fig:explaining_the_state_space_bulk_no_contact}
	\end{subfigure}
	\begin{subfigure}[t]{0.3\textwidth}
		\centering
		\includegraphics[scale=0.12]{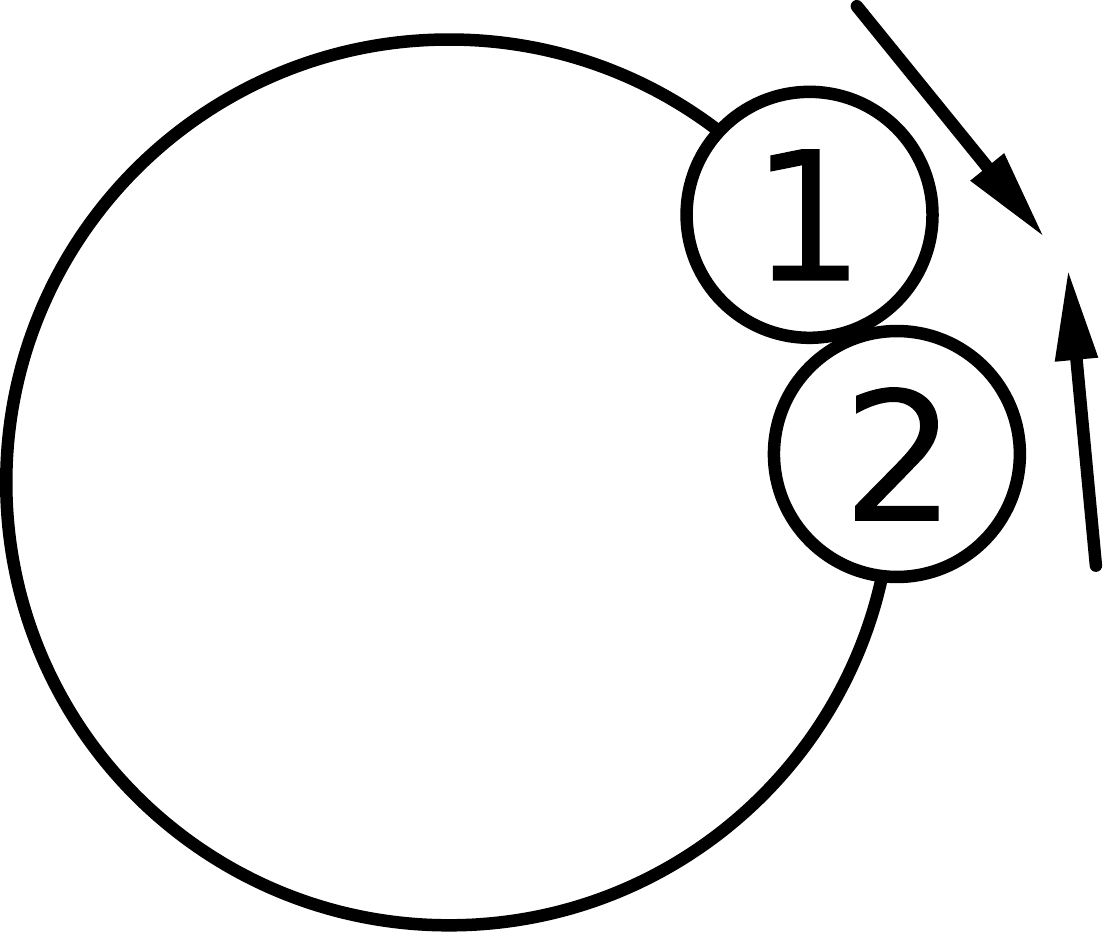}
		\caption{$x = \ell$ and $\nu = *^\ell_{(-1, 1)}\in\partial K$}
		\label{fig:explaining_the_state_space_boundary}
	\end{subfigure}
	\caption{\textcolor{myGreen}{Possible configurations (a, b, c) for the continuous process and a given dynamical variable $\sigma=(-1,1)$ are described by the separation $x$ and the variable $\nu$ encoding the dynamics or jamming state. Bulk configurations form the set  $E_{(-1,1)}$ composed of reentering configurations (a, $(x=0,\nu=(-1,1))$) and free configurations (b, $(0<x<\ell,\nu=(-1,1))$). The jamming configurations (c, $(x=\ell,\nu=*^{\ell}_{(-1,1)})$) form the set $E_{*^\ell_{(-1,1)}}$. As there is no flow mechanism in jamming configurations, there is no need to consider reentering configurations.}}
	\label{fig:explaining_the_state_space}
\end{figure}

\textcolor{myGreen}{Therefore, the process $X_D(t) = (x(t), \nu(t))$ and the process
$X(t) = (x(t), \sigma(t))$ of
definition~\ref{def:continuous_process_min_max_construction} differ
in their state space (respectively $E_D$ and $E$), as their
second component differs. The variable $\sigma = (\sigma^1, \sigma^2)\in \Sigma$
corresponds to the RTP velocities and $\nu\in K = \Sigma \cup \partial K$ encodes both velocity and
jamming information: the velocities $\sigma$ and whether the two
particles are jammed or not. Therefore, such PDMP formalism emphasizes
the role of particular points or sets where flows change or jamming
happens.  Because its construction follows~\cite{davis93}, the
generator of the process $X_D$ is immediately given by theorem 26.14
of~\cite{davis93}. We later introduce a bijection $\iota: E_D \to E$
such that $\iota(X_D(t))$ is the CITP (resp.~CFTP) of subsection
\ref{sec_construction}. Definition~\ref{def:continuous_process_min_max_construction}
and $\iota(X_D(t))$ thus constitute two equivalent constructions of
the same process with different technical advantages. Definition~\ref{def:continuous_process_min_max_construction}
is crucial for designing explicit couplings, whereas the construction
$\iota(X_D(t))$ comes with a precise characterization of the
generator.}

\textcolor{myGreen}{We now do the
  rigorous construction of $X_D$ which closely follows Section 24
  of~\cite{davis93} and adopts its notation.}
  
\textcolor{myGreen}{\emph{State space.}   We first detail the state
  space $E_D=\bigcup_{\nu \in K} E_\nu$ where the family of measurable
  sets $(E_\nu)_{\nu \in K}$ can be divided into bulk and boundary
  sets as follows:}
\begin{itemize}
\item the bulk \textcolor{myGreen}{$\cup_{\nu \in \Sigma} E_\nu$} corresponds to
  states where
  \begin{itemize}    
  \item $x \in \{0, \ell\}$ meaning that the particles are at contact
    and their velocities push them apart (ex:
    figure~\ref{fig:explaining_the_state_space_bulk_contact}). \textcolor{myGreen}{These
    states form up the entry-non-exit points, i.e. the points
    $(x,\nu)$ such that the deterministic dynamics immediately transform them to
    the separated states described just below.}
\item $0 < x < \ell$ meaning that the particles are at positive
  distance and the deterministic dynamics corresponds to the ODE
  $\partial_t x = \sigma^2 - \sigma^1$ (ex: figure
  \ref{fig:explaining_the_state_space_bulk_no_contact}). \textcolor{myGreen}{These states form up the open sets $(E^0_\nu)_{\nu \in \Sigma}$ where $E_\sigma^0 = (0, \ell) \times \left\{\sigma\right\} \text{ for } \sigma \in \Sigma$.}
\end{itemize}
\textcolor{myGreen}{Each set $E_\sigma$ corresponds to a set  $E_\sigma^0$ enlarged by adding the corresponding entry-non-exit-points, so that,
\begin{itemize}
\item
  $E_{(\sigma^1, \sigma^2)} = [0, \ell) \times \{ (\sigma^1, \sigma^2)
  \}$ for $(\sigma^1, \sigma^2) \in \Sigma$ such that
  $\sigma^2 - \sigma^1 > 0$,
\item
  $E_{(\sigma^1, \sigma^2)} = (0, \ell) \times \{ (\sigma^1, \sigma^2)
  \}$ for $(\sigma^1, \sigma^2 ) \in \Sigma$ such that
  $\sigma^2 - \sigma^1 = 0$,
\item
  $E_{(\sigma^1, \sigma^2)} = (0, \ell] \times \{ (\sigma^1, \sigma^2)
  \}$ for $(\sigma^1, \sigma^2) \in \Sigma$ such that
  $\sigma^2 - \sigma^1 < 0$,
\end{itemize}}
\item the boundary $\cup_{\nu \in \partial K} E_\nu$ corresponds to
  states where $x \in \{0, \ell \}$ meaning that the particles are at
  contact and their velocities push them together so they jam
  (ex: figure \ref{fig:explaining_the_state_space_boundary}). \textcolor{myGreen}{In such jammed states,} the distance
  $x$ is constant. \textcolor{myGreen}{These states form up the singletons $E_{*_\sigma^\ell}^0 = \{ \ell \} \times \left\{ *^\ell_\sigma \right\}\text{ for } \sigma \in \Sigma_+$ and $E_{*_\sigma^0}^0 = \{ 0 \} \times \left\{*^0_\sigma \right\} \text{ for } \sigma \in \Sigma_-$.}\\
 \textcolor{myGreen}{ As there is no deterministic flow in jamming configurations which correspond to sticky points of the dynamics, $E_\nu=E_{*_\nu}^0$ for $\nu \in\partial K$. Further details on the extension of the formalism of
Davis to include such sets can be found in \cite{bierkens23}.}
\end{itemize}

\textcolor{myGreen}{To summarize and clarify on the example of the CITP, the space $E_D$ corresponds to the following
  transformation of the state space $E = [0, \ell] \times \{\pm1\}^2$:
  \begin{itemize}
	\item for $\sigma =\pm (1, 1)$, $[0, \ell] \times \{ \sigma \}$ becomes the disjoint union of $E_{*^0_{\sigma}} := \{(0, *^0_{\sigma})\}$, $E_{\sigma} := (0, \ell) \times \{\sigma\}$ and $E_{*^\ell_{\sigma}} := \{(\ell, *^\ell_{\sigma})\}$,
	\item for $\sigma = (1, -1)$ (resp.~$\sigma = (-1, 1)$), $[0, \ell]\times \{\sigma\}$ becomes the disjoint union of $E_{*^x_\sigma} = \{(x, *^x_\sigma)\}$ where $x = 0$ (resp.~$x = \ell$) and $E_\sigma = ([0, \ell] \setminus \{ x \}) \times \{\sigma\}$.
\end{itemize}}

\textcolor{myGreen}{\emph{Dynamics.} Now we detail the deterministic dynamics. On each $E_\nu$ the process follows the deterministic flow induced by the vector field $\mathfrak X_\nu$ where
\begin{itemize}
\item for $\nu \in \Sigma$ the vector field $\mathfrak{X}_\nu$
  defined on $E_\sigma$ corresponds to the ODE
  $\partial_t x = \sigma^2 - \sigma^1$,
\item for $\nu \in \partial K$ the vector field $\mathfrak{X}_\nu$
  defined on $E_\nu$ is null.
\end{itemize}}

\textcolor{myGreen}{\emph{Stochastic jumps.}  The velocities of the particles are independent
Markov jump processes with the transition rates of figure
\ref{fig:single_speed_transition_rates_slowman_2016}
(resp. \ref{fig:single_speed_transition_rates_slowman_2017}). Recall that the states
$(x, (\sigma^1, \sigma^2)), (0, *^0_{(\sigma^1, \sigma^2)})$ and
$(\ell, *^\ell_{(\sigma^1, \sigma^2)})$ all correspond to the first
particle having velocity $\sigma^1$ and the second particle having
velocity $\sigma^2$. We
define}
$\mathcal Q = (q_{\sigma, \sigma'})_{\sigma, \sigma' \in \Sigma}$ as
the discrete generator of the couple $(\sigma^1, \sigma^2)$.  For
example, if the $\sigma^i$ follow figure
\ref{fig:single_speed_transition_rates_slowman_2016} then
\begin{align*}
\mathcal Q =
\bordermatrix{ & \sigma' = (1,1) & \sigma' = (1, -1) & \sigma' = (-1,1) & \sigma' = (-1, -1) \cr
\sigma = (1, 1) & -2 \omega & \omega & \omega & 0 \cr
\sigma = (1, -1) & \omega & -2 \omega & 0 & \omega \cr
\sigma = (-1, 1) & \omega & 0 & -2\omega & \omega \cr
\sigma = (-1, -1) &  0 & \omega & \omega & -2 \omega
}.
\end{align*}
We introduce the following projectors.                                                                 
\begin{definition}
  Define $p_{\sigma^1} : E_D \rightarrow \{-1,1\}$ (resp. $\{-1,0,1\}$) and $p_{\sigma^2}: E_D \rightarrow \{-1,1\}$ (resp. $\{-1,0,1\}$) by
\begin{align*}
p_{\sigma^1} \left(x, \nu\right) &= \sum_{(\sigma^1,\sigma^2)\in\Sigma}\sigma^1 1_{\{(\sigma^1,\sigma^2),*^0_{(\sigma^1,\sigma^2)}, *^\ell_{(\sigma^1,\sigma^2)}\}}(\nu), \\
p_{\sigma^2} \left(x, \nu\right) &= \sum_{(\sigma^1,\sigma^2)\in\Sigma}\sigma^2 1_{\{(\sigma^1,\sigma^2),*^0_{(\sigma^1,\sigma^2)}, *^\ell_{(\sigma^1,\sigma^2)}\}}(\nu).
\end{align*}
and $p_{\sigma}$  by
\begin{align*}
p_{\sigma} \left(x, \nu\right) = (p_{\sigma^1} \left(x, \nu\right),p_{\sigma^2} \left(x, \nu\right)).
\end{align*}
\end{definition}
For $(x,\nu)\in E_D$, we define the jump rate \textcolor{myGreen}{$\lambda:E_D \to \mathbb R_+$} by
$\lambda(x, \nu) = -q_{p_{\sigma}(x,\nu),p_{\sigma}(x,\nu)}$ and \textcolor{myGreen}{the post-jump distribution
kernels for $(x, \nu) \in E_D$ by}
$$
Q((x, \nu); \cdot) =\sum_{\nu' \in K}-\frac{q_{p_{\sigma}(x,\nu),p_{\sigma}(x,\nu')}}{q_{p_{\sigma}(x,\nu),p_{\sigma}(x,\nu)}}
 1_{\{(x, \nu') \in E_D\}} \delta_{(x, \nu')}.
$$
As exactly one of the elements of
$\{ (x, \sigma'); (x, *_{\sigma'}^0); (x, *_{\sigma'}^\ell) \}$ is in the
state space $E_D$, \( Q((x, \nu); \cdot) \) is indeed a probability measure.

\textcolor{myGreen}{\emph{Boundary jumps.} The last issue that has to be dealt with is the possibility that the deterministic flow induced by $\mathfrak X_\nu$ takes the process outside of $E_\nu$. This triggers an instantaneous jump from the boundary $\partial E_\nu$ back into $E_D$, the post-jump distribution of which needs to be specified. In our case, when the deterministic
dynamics leads to a state where the particles are jammed, the process
jumps to the corresponding boundary state. On the exit boundary
$\Gamma^* = \{ (0, \sigma) : \sigma \in \Sigma_+ \setminus \Sigma_-\} \cup \{
(\ell, \sigma) : \sigma \in \Sigma_+ \setminus \Sigma_-\}$, the jump kernel is
given by
\begin{align*}
Q((0, \sigma); \cdot) = \delta_{(0, *_\sigma^0)} \text{ for } \sigma \in \Sigma_+ \setminus \Sigma_- \text{ and }
Q((\ell, \sigma); \cdot) = \delta_{(\ell, *_\sigma^\ell)} \text{ for } \sigma \in \Sigma_+ \setminus \Sigma_-.
\end{align*}}

\textcolor{myGreen}{\subsection{Generator and bijection}}

Many important properties of PDMPs, such as the characterization of
their generator, are true under the set of `standard conditions'
(24.8) of \cite{davis93}. The continuous process clearly satisfies
conditions (1)--(3) and condition (4) follows from proposition (24.6)
of \cite{davis93}. Theorem (26.14) of \cite{davis93} hence yields the
following characterization of the generator.

\begin{prop}[Generator of the continuous process]\label{prop:generator_continuous_process} Let $\mathcal L$ be
  the generator of the continuous process. Its domain $D(\mathcal L)$
  is the set of bounded measurable functions
  $f : E_D \rightarrow \mathbb R$ such that
\begin{itemize}
\item for $\sigma \in \Sigma_+ \setminus \Sigma_-$ the function
  $f( \cdot, \sigma)$ is absolutely continuous on $[0, \ell)$ and
  $\lim_{x \rightarrow \ell} f(x, \sigma) =
  f\left(\ell,*_\sigma^\ell\right)$,
\item for $\sigma \in \Sigma_+ \setminus \Sigma_-$ the function
  $f(\cdot, \sigma)$ is absolutely continuous on $(0, \ell]$ and
  $\lim_{x \rightarrow 0} f(x, \sigma) = f\left(0, *_\sigma^0\right)$.
\end{itemize}

For $f \in D(\mathcal L)$ one has that
\begin{itemize}
\item in the bulk $\cup_{\sigma \in \Sigma} E_\sigma$ the generator
  comprises a transport term and a jump term
\begin{align*}
\mathcal L f(x, (\sigma^1, \sigma^2)) = &\underbrace{(\sigma^2 - \sigma^1) \partial_x f(x, (\sigma^1, \sigma^2))}_{\text{\normalfont transport term}} \\
&\quad + \underbrace{\lambda(x,\nu) \int_{E_D} \left(f(\tilde x, \tilde \nu) - f(x, (\sigma^1, \sigma^2))\right) dQ\left((x, (\sigma^1, \sigma^2)); (\tilde x, \tilde \nu)\right)}_{\text{\normalfont jump term}},
\end{align*}
\item whereas for boundary states
  $(x, \nu) \in \cup_{\nu \in \partial K} E_\nu$ the generator is
  given only by the jump term
$$
\mathcal L f(x, \nu) = \lambda(x, \nu) \int_{E_D} \left(f(\tilde x, \tilde
  \nu) - f(x, \nu)\right) dQ((x, \nu); (\tilde x, \tilde \nu))
$$
\end{itemize}
\end{prop}

To emphasize the difference with definition
\ref{def:continuous_process_min_max_construction}, we denote this
process $X_D(t) = (x(t), \nu(t))$. Formally, definition~\ref
{def:continuous_process_min_max_construction} and the PDMP
construction model the same process, but differ on the choice of
the dynamics variable. Indeed, the vector fields
$(\mathfrak X_\nu)_{\nu\in K}$ have to be continuous, hence the
consideration of $\nu$ and its additional states forming up
$\partial K$. Out of completeness, we rigorously establish such
link, so as to be able to use them both.

\begin{definition}
Define $\iota : E_D \rightarrow [0, \ell] \times \Sigma$ by
$$
\iota(x, \nu) = \left(x, p_\sigma(x,\nu)\right).
$$
\end{definition}

\begin{prop}

  The mapping $\iota$ is a bijection.

\end{prop}

\begin{proof}
  This follows from the fact that for all $x \in [0, \ell]$ and all $\sigma=(\sigma^1,\sigma^2) \in \Sigma$ exactly one element among $(x, \sigma), (x, *_\sigma^0), (x, *_\sigma^\ell)$ is defined and an element of $E_D$. Indeed,
\begin{itemize}
\item if $x \in (0, \ell)$ then $(x, \sigma) \in E_D$ and
  $(x, *^0_\sigma)\not\in E_D$ for $\sigma \in \Sigma_-$,  $(x, *^\ell_\sigma)\not\in E_D$ for $\sigma \in \Sigma_+$,
\item if $x = 0$ then
\begin{itemize}
\item if $\sigma^2 - \sigma^1 > 0$ then $(0, \sigma) \in E_D$ and
  $(0, *^\ell_\sigma) \not \in E_D$,
\item if $\sigma^2 - \sigma^1 \leq 0$ then $(0, *^0_\sigma) \in E_D$ and
  $\{ (0, \sigma); (0, *^\ell_\sigma) \} \cap E_D = \emptyset$,
\end{itemize}
\item if $x = \ell$
\begin{itemize}
\item if $\sigma^2 - \sigma^1 < 0$ then $(\ell, \sigma) \in E_D$ and
  $(\ell, *^0_\sigma)\not \in E_D$,
\item if $\sigma^2 - \sigma^1 \geq 0$ then $(\ell, *^\ell_\sigma) \in E_D$
  and $\{ (\ell, \sigma); (\ell, *^0_\sigma) \} \cap E_D = \emptyset$.
\end{itemize}
\end{itemize}

\end{proof}

\begin{prop} \label{prop_min_max_form}

Let $X_D(t) = (x(t), \nu(t))$ be as in Section~\ref{secdavis}.

\begin{itemize}
\item[(i)] The processes $\sigma^1(t) = p_{\sigma^1}(X_D(t))$ and
  $\sigma^2(t) = p_{\sigma^2}(X_D(t))$ are two independent Markov jump
  processes with the transition rates of figure
  \ref{fig:single_speed_transition_rates_slowman_2016}
  (resp. \ref{fig:single_speed_transition_rates_slowman_2017}).
	\item [(ii)] If one recursively defines
	$$
        \tilde{x}(t) = \max \left[0, \min \left[\ell,\tilde{x}({T_n}) + (\sigma^2(t) -
            \sigma^1(t)) (t - T_n) \right]\right] \text{ for } T_n \le
        t \le T_{n+1}
$$
where $0 = T_0 < T_1 < \cdots$ are the jump times of the couple
$(\sigma^1, \sigma^2)$ then
$\iota(X_D(t)) = (\tilde{x}(t), (\sigma^1(t), \sigma^2(t)))$.
\end{itemize}
\end{prop}

\begin{proof} (i) For all fixed $(s^1, s^2) \in \Sigma$ the function
  $f(x, \nu) = \mathbb P_{(x, \nu)} \left( p_\sigma(X_D(t)) = (s^1,s^2) \right)$ can be computed explicitly using its
  characterization as the solution of a system of differential
  equations (see section 3 of \cite{davis93}). One gets
$$
\mathbb P_{(x, \nu)} \left( p_\sigma(X_D(t))=(s^1,s^2)\right) = \mathbb
P_{p_{\sigma^1}(x, \nu)} \left( \varsigma_1(t) = s_1 \right) \mathbb
P_{p_{\sigma^2}(x, \nu)} \left( \varsigma_2(t) = s_2 \right) \text{
  for all } (s_1, s_2) \in \Sigma
$$
where the $\varsigma_i$ are Markov jump processes with the transition
rates of figure \ref{fig:single_speed_transition_rates_slowman_2016}
(resp. \ref{fig:single_speed_transition_rates_slowman_2017}).

(ii) Follows from the recursive construction in section 23 of
\cite{davis93}.
\end{proof}

\section{From discrete to continuous
  space} \label{sec:from_discrete_to_continuous_space}

\subsection{Convergence in law}

In this section, it is shown that under the right rescaling the discrete
process of \cite{slowman16,slowman17}, i.e. DITP (resp. DFTP),
converges to the continuous process CITP (resp. CFTP).

First, the discrete and continuous process have different state
spaces. The continuous process $X$ takes its values in
$[0, \ell] \times \Sigma$ and one can think of $\{1, \ldots L\}$ as a
discrete subset of the interval $[0, \ell]$ using the injection
$$
i_L(k) := \ell\frac{k - 1}{L - 1}.
$$
This allows one to see the discrete process as taking values in
$[0, \ell] \times \Sigma$ and to consider the process
$(i_L(y^L(t)), (\sigma^1(t), \sigma^2(t)))$ instead of
$(y^L(t), (\sigma^1(t), \sigma^2(t)))$.

It remains to understand how the parameters $\omega$ (resp. $\alpha$
and $\beta$), $\gamma$ and $L$ should scale. Thinking of the discrete
process as an approximation of the continuous process where the
interval $[0, \ell]$ is replaced by its discrete counterpart
$\{1, \ldots, L\}$, the natural scalings
\begin{align*}
(*)\left\{
\begin{tabular}{c}
	$\omega$ (resp. $\alpha$ and $\beta$) constant \\
	$L \rightarrow +\infty$ \\
	$\gamma_L = (L - 1)/\ell$
\end{tabular}
\right.
\end{align*}
correspond to taking increasingly fine discretisations of the position
$y^L(t)$ while preserving the dynamics of the $\sigma^i$ and the
physical velocity of the particles $\ell \gamma / (L - 1)$.

The convergence of Markov processes is often shown using approaches
based on generators or on martingale problems (see
\cite{ethier86}). Here, we take a different approach and instead use a
coupling argument. It requires only basic probabilistic tools and
also provides quantitative estimates for the approximation
procedure. The key insight is that, under the scalings described
above, the stochastic jumps of $y^L(t)$ become both smaller in size
and more frequent in a way that makes them converge to the constant
velocity translations of $x(t)$.

\begin{thm}[Scaling limit]
  \label{thm:quantitative_discrete_continuous_convergence}
  \begin{itemize}
  \item[(i)] There exists a coupling of the continuous process $X$ and
    a sequence of discrete processes $(Y^L)_{L \ge 2}$ with parameters
    scaling as in ($*$) such that for all $T > 0$
    \begin{align*}
      \mathbb{P}\left( \sup_{t \le T} \left| i_L(y^L(t)) - x(t) \right| \ge \epsilon \right) \le \frac{1}{\epsilon} \left( \frac{\ell}{L - 1} + 8 \sqrt{\frac{T \left((\eta T)^2 + 3 \eta T + 1 \right)\ell}{L - 1}} \right)
		\end{align*}
		where $\eta = 2\omega$ (resp. $\eta = 2\max(\alpha,
                \beta)$) for the instantaneous-tumble transition rates of the
                $\sigma^i$ (resp. the finite-tumble transition rates).
              \item[(ii)] One has
		\begin{align*}
                  \Big( (i_L(y^L(t)), (\sigma^1(t), \sigma^2(t))) \Big)_{t\ge 0} \overset{\text{\normalfont Law}}{\longrightarrow} \Big(X(t)\Big)_{t \ge 0}
		\end{align*}
                in the Skorokhod space $\mathcal D[0, +\infty)$.
              \end{itemize}
            \end{thm}

            As the supremum distance is finer than the
            Skorokhod distance, assertion (i) implies that for all
            $T > 0$
$$
\Big( (i_L(y^L(t)), (\sigma^1(t), \sigma^2(t))) \Big)_{t \in [0, T]} \overset{\text{\normalfont Law}}{\longrightarrow} \Big(X(t)\Big)_{t \in [0, T]}
$$
in the sense of random variables taking their values in the Skorokhod
space $\mathcal D[0, T]$. Hence assertion (ii) follows by theorem 16.7
of \cite{billingsley99}. Therefore, the rest of this subsection is
dedicated to the proof of assertion (i).

The jamming mechanism at \( y^L = 1 \) and \( y^L = L \) makes the
control of \( (i_L(y^L(t)) - x(t)) \) challenging. Therefore, we first
consider the following simplified model, without any jamming and taking
values in $\mathbb Z$. For this model, we
derive a uniform bound similar to the one of assertion (i) in theorem
\ref{thm:quantitative_discrete_continuous_convergence}.

\begin{lem} \label{lem:toy_problem} Let $\sigma^1$ and $\sigma^2$ be
  two independent Markov jump processes with the transition rates of
  figure~\ref{fig:single_speed_transition_rates_slowman_2016}
  (resp. \ref{fig:single_speed_transition_rates_slowman_2017}) and let
  $N_1^L$ and $N_2^L$ be two independent Poisson processes with rate
  $\gamma_L$ independent of the $\sigma^i$. \green{Denote $0 = \Theta^{L, i}_0 < \Theta^{L, i}_1 < \cdots$ the jump times of $N_i^L$. If one sets
$$
S^L(t) := \sum_{i = 1, 2} \sum_{k = 1}^{+\infty} 1_{\{ \Theta^{L,i}_k \le t \}}  (-1)^i \sigma^i(\Theta^{L,i}_k) \text{ and } \ I(t) := \int_0^t (\sigma^2(s) - \sigma^1(s))ds
$$}
then one has
$$
\mathbb P \left( \sup_{t \le T} \left| \frac{1}{\gamma_L} S^L(t) - I(t)\right| \ge \epsilon \right) \le \frac{2}{\epsilon} \sqrt{\frac{T}{\gamma_L}}.
$$
\end{lem}

The proof is included for the sake of completeness.

\begin{proof}
  First, $(S^L(t), (\sigma^1(t), \sigma^2(t)))$ is a
  $(\mathbb Z \times \Sigma)$-valued Markov jump process and its
  generator $\mathcal L^L$ is given by
  \begin{multline*}
\mathcal L^L f(z, (\sigma^1, \sigma^2)) = \gamma_L \big(f(z - \sigma^1,
(\sigma^1, \sigma^2)) +  f(z + \sigma^2, (\sigma^1, \sigma^2)) - 2
f(z, (\sigma^1, \sigma^2)\big) \\+ \sum_{(s^1, s^2) \in \Sigma}
q_{(\sigma^1, \sigma^2), (s^1, s^2)} f(z, (s^1, s^2))
\end{multline*}
where
$\mathcal Q = (q_{\sigma, \sigma'})_{\sigma, \sigma' \in \Sigma}$ is
the generator of the couple $(\sigma^1, \sigma^2)$, as defined in the
previous Section. Hence Dynkin's formula implies that
$S^L(t) - \gamma_L \int_0^t (\sigma^2(s) - \sigma^1(s)) ds$ is a
martingale and by Doob's martingale inequality
\begin{align*}
&\mathbb P \left( \sup_{t \le T} \left| \frac{1}{\gamma_L} S^L(t) - I(t) \right| \ge \epsilon \right) \\
&\qquad\qquad\qquad\le \frac{\mathbb E \left| S^L(T) - \gamma_LI(T) \right|}{\epsilon \gamma_L}, \\
&\green{\qquad\qquad\qquad\le \frac{\sum_{i = 1, 2} \mathbb E \left| \sum_{k = 1}^{+\infty} 1_{\{\Theta^{L, i}_k \le T\}}\sigma^i(\Theta^{L,i}_k) - \gamma_L\int_0^T \sigma^i(s) ds \right|}{\epsilon \gamma_L}, }\\
&\green{\qquad\qquad\qquad\le \frac{2}{\epsilon \gamma_L} \sqrt{\mathbb E \left( \sum_{k = 1}^{+\infty} 1_{\{\Theta^{L, 1}_k \le T\}}\sigma^1(\Theta^{L,1}_k) - \gamma_L \int_0^T \sigma^1(s) ds \right)^2}.}
\end{align*}

It remains to show that
\green{$\mathbb E \left( \sum_{k = 1}^{+\infty} 1_{\{\Theta^{L, 1}_k \le T\}}\sigma^1(\Theta^{L,1}_k) - \gamma_L \int_0^T \sigma^1(s) ds \right)^2 \le
T\gamma_L$}. Let $(T_n)_{n \ge 0}$ be the jump times of the couple
$(\sigma^1, \sigma^2)$. The process
\green{$M(t) = \sum_{k = 1}^{+\infty} 1_{\{\Theta^{L, 1}_k \le t\}}\sigma^1(\Theta^{L,1}_k) - \gamma_L \int_0^t \sigma^1(s) ds$} is also a martingale by Dynkin's
formula so that repeated use of the optional stopping theorem yields
\begin{align*}
\mathbb E \left[ M(T)^2 \right] = \sum_{k = 1}^N \mathbb E \left[ (M(T_k \wedge T) - M(T_{k-1} \wedge T))^2 \right] + \mathbb E \left[ (M(T) - M(T_{N} \wedge T))^2 \right]. 
\end{align*}
Taking $N \rightarrow +\infty$ and using the dominated convergence
theorem one gets
$$\mathbb E \left[ (M(T) - M(T \wedge T_N))^2 \right] \rightarrow 0$$
and
$\mathbb E \left[ M(T)^2 \right] = \sum_{k = 1}^{+\infty} \mathbb E
\left[ (M(T_k \wedge T) - M(T_{k-1} \wedge T))^2 \right]$ so that
\begin{align*}
\mathbb E 
\left[ M(T)^2 \right]&= \sum_{k = 1}^{+\infty} \mathbb E \left[ \sigma^1(T_{k-1} \wedge T)^2 (N_1^L(T_{k}\wedge T) - N_1^L(T_{k-1} \wedge T) - \gamma_L (T_k \wedge T - T_{k-1} \wedge T))^2\right], \\
&\le \sum_{k = 1}^{+\infty}  \mathbb E \left[ \left(N_1^L(T_{k}\wedge T) - N_1^L(T_{k-1} \wedge T) - \gamma_L (T_k \wedge T - T_{k-1} \wedge T) \right)^2 \right] \\
&= \mathbb E \left[\left( N_1^L(T) - \gamma_L T\right)^2\right] = \gamma_L T.
\end{align*}
\end{proof}
\begin{rem}
It is easy to see that one can also directly obtain an upper bound in
the previous lemma such as $8T/(\epsilon^2\gamma_L)$ (in fact even
exponential inequality can be derived). This bound worsens with
respect to $T$, but exhibits improved behavior with respect to
$\gamma_L$.
\end{rem}

Formally, for $\gamma_L \rightarrow +\infty$, the jumps of
$S^L/\gamma_L$ have decreasing size $1/\gamma_L$ but
increasing rate $\gamma_L$ so that the physical velocity
$\gamma_L / \gamma_L = 1$ remains constant. Lemma
\ref{lem:toy_problem} rigorously implies that the stochastic jumps of
$S^L/\gamma_L$ converge to the constant velocity translations of
$I(t)$ with the supremum norm,
thanks to Doob's martingale inequality.\medskip

We now introduce the following coupling to prove the convergence of
the discrete process to the continuous process, i.e. assertion (i) of
the theorem \ref{thm:quantitative_discrete_continuous_convergence}, by
drawing on the lines of the proof of lemma \ref{lem:toy_problem}.

\begin{definition}[Discrete-continuous coupling]
  \label{def:discrete_continuous_coupling}
  Let $X(t) = (x(t), (\sigma^1(t), \sigma^2(t)))$ be a continuous
  process as defined in definition
  \ref{def:continuous_process_min_max_construction}, with fixed
  parameters $\omega$ (resp. $\alpha$ and $\beta$) and $\ell$, and
  $\sigma^1(t)$ and
  $\sigma^2(t)$ two independent Markov jump
  processes with the transition rates of figure
  \ref{fig:single_speed_transition_rates_slowman_2016}
  (resp. \ref{fig:single_speed_transition_rates_slowman_2017}). Denote
  by $0 = T_0 < T_1 < \cdots$ the jump times of the couple
  $(\sigma^1, \sigma^2)$. For $L \ge 2$, let $N_1^L$ and $N_2^L$ be
  two independent Poisson processes with parameter
  $\gamma_L = (L - 1)/\ell$ independent of the $\sigma^i$.

  Recursively define $y^L$ by setting
\[
y^L(0) = \left\lfloor (L-1)x(0)/\ell\right\rfloor+ 1 \ \text{and}\  	z^L(t) := y^L(T_k) + \sum_{i = 1, 2} (-1)^i \sigma^i(T_k)
  \left[ N^L_i(t) - N^L_i(T_k) \right],
\]
  \begin{itemize}    
  \item if $(\sigma^1(T_k), \sigma^2(T_k)) \neq \pm (1, 1)$
$$
y^L(t) = \max \left[1,  \min \left[ L, z^L(t)  \right] \right] \text{ for } t \in [T_k, T_{k+1}],
$$
\item if $(\sigma^1(T_k), \sigma^2(T_k)) = \pm (1, 1)$
$$
y^L(t) = p_L \left( z^L(t) \right) \text{ for } t \in
[T_k, T_{k+1}]
$$
where $p_L : \mathbb Z \rightarrow \{1, \ldots, L\}$ is given by
$$
p_L(k + 2 m L) =
\left\{
\begin{array}{cl}
k &\text{ for } m \in \mathbb Z \text{ and } 1 \le k \le L, \\
2L + 1 - k &\text{ for } m \in \mathbb Z \text{ and } L + 1 \le k \le 2L.
\end{array}
\right.
$$
\end{itemize}

Finally, set $Y^L(t) = (y^L(t), (\sigma^1(t), \sigma^2(t)))$.
\end{definition}

The construction in definition \ref{def:discrete_continuous_coupling}
of $Y^L$ using $p_L$ differs from definition
\ref{def:discrete_process}. The goal of this coupling is to map the
transitions of the DITP (resp.~DFTP) to the jamming-free random walk
\( z^L \), that was already studied in the toy model of
lemma~\ref{lem:properties_of_the_discrete_continuous_coupling}.  The
difference lies in the case \( \sigma^1 = \sigma^2 = \pm 1 \). A
definition through the $\max(\min(\cdot))$ definition is not correct
at jamming, as the $y^L$ is evaluated and fixed at its $t=T_k$
value. However, the definition through $p^L$ ensures a correct mapping
by exploiting the indistinguishability of the particles (i.e. the
Poisson process $N^L_1$ (resp. $N^L_2$) does not correspond to the
Poisson clock linked to particle $1$ (resp. $2$)). In more details, in
such case, the process $y^L$ comes down to a symmetric random walk
(see figure~\ref{fig:yL_transitions}), that is mapped to the symmetric
random walk $z^L$ in this manner:
\begin{itemize}
\item Case \( p_L(z^L) = y^L \notin \{ 1, L \} \); If \( z^L \) jumps
  left (resp.~right) with rate \( \gamma_L \), then \( p_L \) maps this
  to the left (resp.~right) jump of \( y^L \).
\item Case \( p_L(z^L) = y^L \in \{ 1, L \} \), say
  \( z^L = y^L = L \); If \( z^L \) jumps to \( L - 1 \) with rate
  \( \gamma_L \), then \( p_L \) maps this to the jump of \( y^L \) to
  \( L - 1 \). If \( z^L \) jumps to \( L + 1 \), the value of
  \( p_L(z^L) \) stalls at $L$ and \( y^L \) does not jump.
        \end{itemize}
        While the actual proof of the mapping is required, such
        mapping enables a direct derivation of the convergence
        bound. Indeed, it now appears possible
        to bound $\sup_{t \le T} \left| i_L(y^L(t)) - x(t) \right|$ in
        terms of $\sup_ {t \le T} |N_i^L(t) - \gamma_L t|$. Hence,
        because $(N_i^L(t) - \gamma_L t)$ is a martingale, a uniform
        bound similar to lemma \ref {lem:toy_problem} is
        attainable.
        
\begin{figure}
\centering
\includegraphics[width=6cm]{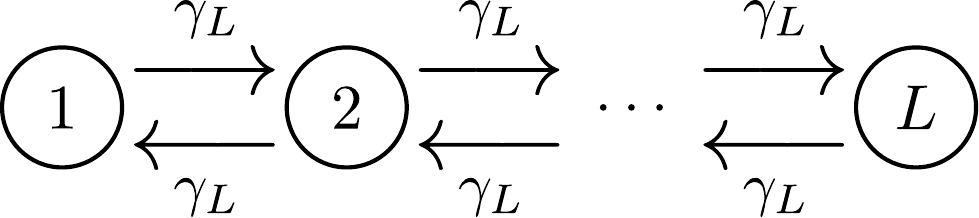}
\caption{Transition rates of the DITP (resp.~DFTP) conditioned on
\( \sigma^1 = \sigma^2 = \pm 1 \)}
\label{fig:yL_transitions}
\end{figure}

{\color{black}
\begin{rem}\label{rem:no_proof_of_discrete_continuous_convergence_in_general}
  If $\sigma^1, \sigma^2$ follow general transition rates as
  in~\cite{hahn23}, differing from
  figure~\ref{fig:single_speed_transition_rates_slowman_2017_and_2017},
  then it is possible that neither $k \mapsto \max(L, \min(1, k))$ nor
  $p_L$ map a $\mathbb Z$-valued random walk to the rates of
  $y^L$. Such an explicit mapping is however required for the proof of
  theorem~\ref{thm:quantitative_discrete_continuous_convergence}. We
  conjecture that the continuous-space limit remains valid and can be
  shown using an adapted mapping or, in a more general treatment,
  using the generator approach~\cite{ethier86}. As previously emphasized, this
  alternative generator approach does not yield a quantitative
  convergence speed and also comes with technical challenges arising from
  the boundaries of the state space, leading to a complex study of the domain of the generators considered.
\end{rem}
}

\begin{lem} \label{lem:properties_of_the_discrete_continuous_coupling}

  Let $X$ and the sequence of stochastic processes $(Y^L)_{L \ge 2}$ be
  as in definition~\ref{def:discrete_continuous_coupling}.
\begin{itemize}
\item[(i)] Each $Y^L$ is a DITP (resp. DFTP).
\item[(ii)] Let $0 = T_0 < T_1 < \cdots$ be the jump times of the
  couple $(\sigma^1, \sigma^2)$. For all $t \ge 0$ one has
\begin{align*}
\left| i_L(y^L(t)) - x(t) \right| \le &\left| i_L(y^L(0)) - x(0) \right| + \frac{1}{\gamma_L} \sum_{k = 1}^{n(t)} \left| D_k(t) \right|
\end{align*}
where $n(t) = 1 + \sum_{k = 1}^{+\infty} 1_{\{T_k \le t\}}$ and
$$
D_k(t) = \sum_{i = 1, 2} (-1)^i \sigma^i(T_{k-1}) \Big[ N_i^L(T_k \wedge t) - N_i^L(T_{k-1} \wedge t) - \gamma_L \left( T_k \wedge t - T_{k-1} \wedge t \right)\Big].
$$
\end{itemize}
\end{lem}

\begin{proof}
 (i) First \green{we} prove that the
  definition~\ref{def:discrete_continuous_coupling} indeed provides
  a coupling, i.e. $Y^L$ alone is indeed a DITP (or DFTP).
  
  Let $L \ge 2$ be arbitrary but fixed. Let $Y^L$ be as in
  definition \ref{def:discrete_continuous_coupling} and let $\sigma^i$
  and $N_i^L$ be the corresponding velocities and Poisson processes. Set
  $\tilde Y^L = \mathcal{Y}(\sigma^1, \sigma^2, N_1^L, N_2^L, L)$.

  To show that $Y^L$ and $\tilde Y^L$ have the same law, it suffices to show
  that $y^L$ and $\tilde y^L$ are identically distributed conditional on the
  velocities $(\sigma^1, \sigma^2)$. Let $s :
  [0, +\infty) \rightarrow \Sigma$ be an arbitrary but fixed càdlàg function.
  In the remainder of the proof of (i), take the conditional probability
  $\mathbb Q = \mathbb P \left(\cdot \, | \, (\sigma^1, \sigma^2) = s\right)$
  to be the reference probability so that functionals of the $\sigma^i$\green{, such
  as jump times,} are deterministic. Under the probability $\mathbb Q$ both
  $y^L$ and $\tilde y^L$ are time-inhomogeneous Markov processes with the
  same initial distribution so that it suffices to prove that their
  transition probabilities are the same.

  Recall that $0 = T_0 < T_1 < \cdots$ are the jump times of the
  couple $(\sigma^1, \sigma^2)$. For $t \in [T_k, T_{k+1}]$,
  \begin{itemize} \item if $(\sigma^1(T_k), \sigma^2(T_k)) \neq \pm(1, 1)$
   then \begin{align*} y_L(t) &= \max \left[0, \min \left[L, y^L(T_k) + \sum_
   {i = 1, 2} (-1)^i \sigma^i(T_k) (N_i^L(t) - N_i^L
   (T_k))\right]\right],\\ \tilde y_L(t) &= \max \left[0, \min \left
   [L, \tilde y^L(T_k) + \sum_{i = 1, 2} (-1)^i \sigma^i(T_k) (N_i^L(t) - N_i^L
   (T_k))\right]\right], \end{align*} so that the transition rates are the same
   on the time interval $[T_k, T_{k+1}]$,

 \item if $(\sigma^1(T_k), \sigma^2(T_k)) = \pm(1, 1)$ then $z^L(t)$
   is a symmetric random walk on \( \mathbb Z \) with rate
   \( \gamma_L \), the transition rates of which are denoted
   \( \hat\kappa \). Further denote \( \tilde\kappa \) the rates of
   the DITP (resp.~DFTP) conditionned on
   \( \sigma^1 = \sigma^2 = \pm 1 \) (see
   figure~\ref{fig:yL_transitions}).  For all
   \( y_1, y_2 \in \{ 1, \ldots, L \} \) one has
   \[ \sum_{z_2 \in p_L^ {-1}(y_2)} \hat\kappa(z_1 \to z_2) =
     \tilde\kappa(y_1 \to y_2) \text{ for all } z_1 \in p_L^{-1}(y_1) \] so
   \cite[Theorem 2.4]{ball93} implies that
   \( y^L(t) = p_L(z^L(t)) \) is indeed a Markov jump process
   and \cite[Theorem 2.3 (i)]{ball93} implies that its rates identfies with $\tilde\kappa$.
\end{itemize}

(ii) The velocities $\sigma^i$ are constant on the time interval
$[T_{n(t)}, t]$.
	\begin{itemize}
		\item If $\sigma^1(T_{n(t)}) \neq \sigma^2(T_{n(t)})$ one has
		\begin{align*}
                  \left| i_L( y^L(t)) - x(t) \right| &= \Bigg| \max \left[0, \min\left[\ell, i_L( y^L(T_{n(t)})) + \frac{1}{\gamma_L}  \sum_{i = 1, 2} (-1)^i \sigma^i(T_{n(t)}) (N^L_i(t) - N^L_i(T_{n(t)})) \right] \right] \\
                                                     &\quad\quad\quad- \max\left[0, \min \left[\ell, x(T_{n(t)}) + (\sigma^2(T_{n(t)}) - \sigma^1(T_{n(t)}))(t - T_{n(t)}) \right]\right] \Bigg|, \\
                                                     &\le \left| i_L( y^L(T_{n(t)})) - x(T_{n(t)})\right| \\
                                                     &\quad\quad\quad+ \frac{1}{\gamma_L}\left| \sum_{i = 1, 2} (-1)^i \sigma^i(T_{n(t)}) \left( N^L_i(t) - N^L_i(T_{n(t)}) - \gamma_L(t - T_{n(t)}) \right) \right|.
		\end{align*}
              \item If $\sigma^1(T_{n(t)}) = \sigma^2(T_{n(t)}) = 0$
                one has
                $\left| i_L( y^L(t)) - x(t) \right| = \left| i_L(
                  y^L(T_{n(t)})) - x(T_{n(t)}) \right|$.
              \item If
                $\sigma^1(T_{n(t)}) = \sigma^2(T_{n(t)}) = \pm 1$,
                because $p_L$ is $1$-Lipschitz and $i_L$ is
                $\gamma_L^{-1}$-Lipschitz, one has
		\begin{align*}
                  \left|i_L( y^L(t)) - x(t)\right| &\le \left| i_L( y^L(T_{n(t)})) - x(T_{n(t)}) \right| + \left| i_L( y^L(t)) - i_L( y^L(T_{n(t)})) \right|, \\
                                                   &= \left| i_L( y^L(T_{n(t)})) - x(T_{n(t)}) \right| \\
                                                   &\quad+ \frac{1}{\gamma_L} \left| p_L\left( y^L(T_{n(t)}) + \sum_{i = 1, 2} (-1)^i \sigma^i(T_{n(t)}) (N^L_i(t) - N^L_i(T_{n(t)}))\right)  - p_L\left( y^L(T_{n(t)})\right) \right|, \\
                                                   &\le \left| i_L( y^L(T_{n(t)})) - x(T_{n(t)}) \right| \\
                                                   &\quad\quad\quad+ \frac{1}{\gamma_L}\left| \sum_{i = 1, 2} (-1) \sigma^i(T_{n(t)}) \left( N^L_i(t) - N^L_i(T_{n(t)}) - \gamma_L (t - T_{n(t)}) \right)\right|.
		\end{align*}
              \end{itemize}
	
	So in any case
	$$
	\left|i_L( y^L(t)) - x(t)\right| \le \left| i_L(
          y^L(T_{n(t)})) - x(T_{n(t)}) \right| + \frac{1}{\gamma_L}
        \left| D_{n(t)}(t) \right|
	$$
	and iterating the same argument yields the desired bound.
\end{proof}

The proof of theorem
\ref{thm:quantitative_discrete_continuous_convergence} (i) is now possible. The
 fact that the upper bound (ii) of lemma
\ref{lem:properties_of_the_discrete_continuous_coupling} contains the term
 $\sum \left| D_k \right|$ instead of $\left| \sum D_k \right|$ leads to a
 slightly worse bound compared to the toy problem. In particular, unlike for
 the toy problem, the bound of theorem
\ref{thm:quantitative_discrete_continuous_convergence} (i) depends on $\omega$
 (resp. $\alpha$ and $\beta$).

\begin{proof}[Theorem~\ref{thm:quantitative_discrete_continuous_convergence} (i)] Let $X$
  and $(Y^L)_{L \ge 2}$ be as in definition
  \ref{def:discrete_continuous_coupling}. Keeping the same notation as
  in lemma~\ref{lem:properties_of_the_discrete_continuous_coupling} one
  has for $t \le T$
	\begin{align*}
	\left|D_k(t)\right| &= \left| \sum_{i = 1, 2} (-1)^i \sigma^i(T_k) \Big[ N_i^L(T_k \wedge t) - N_i^L(T_{k-1} \wedge t) - \gamma_L \left( T_k \wedge t - T_{k-1} \wedge t \right)\Big] \right|, \\
	&\le \sum_{i = 1, 2} \left( \left|N_i^L(T_k \wedge t) - \gamma_L (T_k \wedge t) \right| + \left|N_i^L(T_{k-1} \wedge t) - \gamma_L (T_{k-1} \wedge t) \right|\right), \\
	&\le 2 \sum_{i = 1, 2} \sup_{\theta \le T} \left| N_i^L(\theta) - \gamma_L \theta \right|. 
	\end{align*}
	
	So claim (ii) of lemma
        \ref{lem:properties_of_the_discrete_continuous_coupling}
        implies
	$$
	\sup_{t \le T} \left|i_L( y^L(t)) - x(t) \right| \le \left| i_L( y^L(0)) - x(0) \right| + \frac{2 n(T)}{\gamma_L} \sum_{i = 1, 2} \sup_{t \le T} \left|  N_i^L(t) - \gamma_L t \right|.
	$$
	
	Using Markov's inequality first and then the Cauchy-Schwarz
        inequality one gets
	\begin{align*}
	\mathbb P \left( \sup_{t \le T} \left| i_L( y^L(t)) - x(t)\right| \ge \epsilon \right) &\le \frac{1}{\epsilon} \left( \mathbb{E} \left| i_L( y^L(0)) - x(0)\right|  + \frac{4}{\gamma_L} \mathbb E \left[ n(T) \sup_{t \le T} \left| N^L_1(t) - \gamma_L t \right| \right]\right), \\
	&\le \frac{1}{\epsilon} \left( \frac{\ell}{L - 1} + \frac{4}{\gamma_L}\sqrt{\mathbb E \left[ n(T)^2 \right]} \sqrt{\mathbb{E} \left[ \sup_{t \le T} \left( N^L_1(t) - \gamma_L t\right)^2 \right]} \right).
	\end{align*}
	
	It follows from Doob's maximal inequality that
	$$
	\mathbb{E} \left[ \sup_{t \le T} \left( N^L_1(t) - \gamma_L t\right)^2 \right] \le 4 \mathbb E \left[ (N^L_1(T) - \gamma_L T)^2 \right] = 4 \gamma_L T.
	$$
	
	Finally, a coupling argument shows that one has the stochastic
        domination
	$$
	\mathbb P \left( n(T) \ge q \right) \le \mathbb P \left(1 + \tilde n \ge q \right) \text{ for all } q \ge 0,
	$$
	where $\tilde n$ is a Poisson random variable with parameter
        $\eta T$. Hence
	$$
	\mathbb E \left[ n(T)^2 \right] \le \mathbb E \left[ \left(1 +
            \tilde n\right)^2 \right] = \mathbb E \left[ \tilde
          n^2\right] + 2 \mathbb E \left[ \tilde n\right] + 1 = \left(
          \eta T \right)^2 + 3\eta T + 1
	$$
	and the desired result follows.
\end{proof}

\subsection{Convergence of the discrete invariant measures}

Let $X$ and $(Y^L)_{L \ge 2}$ be as in definition
\ref{def:discrete_continuous_coupling}. Because $Y^L$ is irreducible
and has finite state space, it has a unique invariant measure $\pi_L$
and there exist $C_L, \rho_L > 0$ such that
$$
\sup_{\mu_L} \left\| \mu_L P^L_t - \pi_L \right\|_\text{TV} \le C_L e^{-\rho_L t} \text{ for all } t \ge 0
$$
where $(P^L_t)$ is the semigroup associated to $Y^L$. Similarly, the
continuous process satisfies the Doeblin property ensuring the
existence of a unique invariant measure $\pi$ as well as exponential
convergence towards that measure. However, the bounds on the speed of
convergence towards the invariant measure are quantitatively poor.

\begin{prop}[Doeblin]
Let \( X(t) \) be the CITP (resp.~CFTP) of definition~\ref{def:continuous_process_min_max_construction} and \( (P_t) \) its semigroup.
\begin{itemize}
	\item[(i)] One has
	\[
		\inf_{(x, \sigma) \in [0, \ell] \times \Sigma} \mathbb{P}_{(x, \sigma)} \Big( X({1 + \ell/2}) = \left(0, (1, -1) \right) \Big) > 0.
	\]
	\item[(ii)] The process \( X(t) \) admits a unique invariant probability $\pi$ and
there exist $C, \rho > 0$ such that
$$
\sup_\mu \, \tv{\mu P_t - \pi} \le C e^{-\rho t} \text{ for all } t\ge
0,
$$
where the supremum is taken over all probability measures \( \mu \).
\end{itemize}
\end{prop}

\begin{proof} (i) Let \( (x_0, \sigma_0) \in [0, \ell] \times \Sigma \) be arbitrary but fixed. The Markov property yields
\begin{align*}
\mathbb{P}_{(x_0, \sigma_0)} \Big( X({1 + \ell/2}) = (0, (1, -1)) \Big) \ge \mathbb{P}_{(x_0, \sigma_0)} \Big( \sigma(1) = (1, -1) \text{ and } X({1 + \ell/2}) = (0, (1, -1)) \Big), \\
\ge \left[ \inf_{(x, \sigma) \in [0, \ell] \times \Sigma} \mathbb{P}_{(x, \sigma)} \Big( \sigma(1) = (1, -1) \Big) \right] \times \left[ \inf_{y \in [0, \ell]} \mathbb{P}_{(y, (1, -1))} \Big( X({\ell/2}) = (0, (1, -1)) \Big) \right].
\end{align*}
Because the velocity process $\sigma(t)$ is an irreducible Markov jump process, one has 
$$
\inf_{(x, \sigma) \in [0, \ell] \times \Sigma} \mathbb{P}_{(x, \sigma)} \left( \sigma(1) = (1, -1) \right) > 0.
$$
For all \( y \in [0, \ell] \), starting from the intial state \( (y, (1, -1)) \) at time \( t = 0 \), one has
\[
	\text{there is no jump before time } t = \ell/2 \quad \implies \quad X(\ell/2) = (0, (1, -1)),
\]
so that
$$
\inf_{y \in [0, \ell]} \mathbb{P}_{(y, (1, -1))} \Big( X({\ell/2}) = (0, (1, -1)) \Big) \ge e^{-\lambda \ell} > 0,
$$
where \( \lambda = \omega \) for the CITP (resp.~\( \lambda = \alpha \) for the CFTP).

(ii) In light of the discussion above, the discrete-time Markov kernel \( Q = P_{1 + \ell / 2} \) satisfies the minorization condition
\[
	\delta_{(x, \sigma)} Q (\cdot) \ge \xi(\cdot) \text{ for all } (x, \sigma) \in [0, \ell] \times \Sigma,
\]
with \( \xi = \epsilon \delta_{(0, (-1, 1))} \) and \( \epsilon = \inf_{(x, \sigma) \in [0, \ell] \times \Sigma} \mathbb P_{(x, \sigma)} \Big( X(1 + \ell/2) = (0, (1, -1)) \Big) > 0 \).

Hence~\cite[Th.~8.7]{benaim22course} implies that \( Q \) admits a unique invariant probability measure \( \pi \) and
\[
	\tv{\mu Q^n - \pi } \le (1 - \epsilon)^n \tv{\mu - \pi},
\]
for all probability measures \( \mu \).

The proccess \( X(t) \) admits at least one invariant probability by~\cite[Prop.~4.56]{benaim22course}. Because all invariant probabilities of \( X(t) \) are invariant probabilities of \( Q \), one has that \( \pi \) is invariant for \( X(t) \) and that it is unique. Finally
\[
	\tv{\mu P_t - \pi} \le \tv{\mu P_{\left\lfloor t / T\right\rfloor T} - \pi} \le (1 - \epsilon)^{\left\lfloor t / T\right\rfloor} \tv{\mu - \pi} \le \frac1{1 - \epsilon} e^{-\frac{\log\left(\frac1{1 - \epsilon}\right)}{T} t} \tv{\mu - \pi}.
\]
\end{proof}

The convergence, under an appropriate rescaling, of $Y^L$ towards $X$
thus invites the question: does $\pi_L$ converge to $\pi$ ? This
amounts to interchanging the $L \rightarrow +\infty$ and the
$t \rightarrow +\infty$ limits and was a crucial implicit assumptions
in \cite{slowman16,slowman17}. This kind of limit interchange
requires some sort of uniformity result, which in this case takes
the form of proposition \ref{prop:uniform_discrete_mixing}.

\begin{prop}\label{prop:uniform_discrete_mixing} There exists
  $C, \rho > 0$ such that
	$$
	\sup_{L} \sup_{\mu_L} \, \tv{\mu_L P^L_t - \pi_L} \le C e^{-\rho t} \text{ for all } t \ge 0.
	$$
\end{prop}

The following \green{corollary} is a direct consequence of proposition
\ref{prop:uniform_discrete_mixing}.

\begin{cor}\label{cor:limit_interchange} Define
   $i^\sigma_L : \{1, \ldots, L\} \times \Sigma\rightarrow [0,\ell]\times\Sigma$ as
$$
i^\sigma_L(y, (\sigma^1,\sigma^2)) = \left( i_L(y), (\sigma^1,\sigma^2) \right).
$$
  One has
  $$W_1(i^{\sigma}_L \# \pi_L, \pi) \rightarrow 0.$$
\end{cor}

\begin{proof} Theorem
  \ref{thm:quantitative_discrete_continuous_convergence} together with
  the boundedness of $E = [0, \ell] \times \Sigma$ implies that for
  all $t \ge 0$
$$
\lim_{L \rightarrow +\infty} W_1(i^\sigma_L \# \left( \left( s_L \# \pi \right)P^L_t\right), \pi P_t) = 0
$$
where $s_L : E \rightarrow \{1, \ldots, L\} \times \Sigma$ is given by
$$
s_L(x, (\sigma^1,\sigma^2)) = \left( \left\lfloor (L - 1) x /\ell \right\rfloor  + 1, (\sigma^1,\sigma^2) \right).
$$
Since $\| X - X' \| \le \sqrt{4^2 + 4^2 + \ell^2} =: A$ for all
$X, X' \in [0, \ell] \times \Sigma$, one has
$W_1(\mu, \nu) \le A \tv{\mu - \nu}$ for all measure $\mu, \nu$ on
$[0, \ell] \times \Sigma$. One gets for all $t \ge 0$
\begin{align*}
W_1(i_L^\sigma \# \pi_L, \pi) &\le W_1(i_L^\sigma \# \pi_L, i_L^\sigma \# \left( \left( s_L \# \pi \right)P^L_t\right)) + W_1(i_L^\sigma \# \left( \left( s_L \# \pi \right)P^L_t\right), \pi P_t), \\
&\le A \tv{i_L^\sigma \# \pi_L - i_L^\sigma \# \left( \left( s_L \# \pi \right)P^L_t\right)} + W_1(i_L^\sigma \# \left( \left( s_L \# \pi \right)P^L_t\right), \pi P_t),\\
&\le A \tv{\pi_L - \left( s_L \# \pi \right)P^L_t} + W_1(i_L^\sigma \# \left( \left( s_L \# \pi \right)P^L_t\right), \pi P_t), \\
&\le A C e^{-\rho t} + W_1(i_L^\sigma \# \left( \left( s_L \# \pi \right)P^L_t\right), \pi P_t).
\end{align*}

Hence
$\limsup_{L \rightarrow +\infty} W_1(i_L^\sigma \# \pi_L, \pi) \le AC
e^{-\rho t}$. The result follows since $t \ge 0$ is arbitrary.
\end{proof}

\begin{rem}\label{rem:non_constructive_discrete_continuous_limit} The set \( E
 = [0, \ell] \times \Sigma \) is compact so \( \mathcal P(E) \) is compact w.r.t.~convergence in law. Furthermore,
 convergence in the Wasserstein metric is equivalent to convergence in law
 and convergence of the first moment hence the sequence \( (\pi^L) \) lives
 in a compact set. Thus, it suffices to prove that the invariant measure
\( \pi \) is its unique accumulation point to show
\[ W_1(i_L^\sigma \# \pi_L, \pi) \rightarrow 0.
\]

Let \(i_{L_k}^\sigma \# \pi^{L_k} \to \pi^\infty \) then the proof of
Theorem~\ref{thm:quantitative_discrete_continuous_convergence} shows \( i_
{L_k}^\sigma \# (\pi^{L_k} P^{L_k}_t) \to \pi^\infty P_t \) so that
\[ \pi^\infty P_t = \lim i_{L_k}^\sigma \# (\pi^{L_k} P^{L_k}_t) = \lim i_
 {L_k}^\sigma \# \pi^{L_k} = \pi^\infty.
\] Hence \( \pi^\infty = \pi \) is the unique invariant measure of the
 continuous process.
\end{rem}

The remainder of this section is dedicated to the proof of proposition
\ref{prop:uniform_discrete_mixing} using a coupling argument. Coupling two
 discrete processes $Y^L = (y^L, (\sigma^1, \sigma^2))$ and $\tilde Y^L =
 (\tilde y^L, (\tilde \sigma^1, \tilde \sigma^2))$ comes down to coupling their
 velocities and their Poisson clocks. The latter are taken to be identical and the
 velocities are independent until they meet for the first time and are then set
 identical. Let us formalize this in the following definition.

\begin{definition}[Discrete-discrete coupling] \label
 {def:discrete_discrete_coupling}

  Let $Y^L_0 = (y^L_0, (\sigma^1_0, \sigma^2_0))$ and
  $\tilde Y^L_0 = (\tilde y^L_0, (\tilde \sigma^1_0, \tilde
  \sigma^2_0))$ in $\{1, \ldots, L\} \times \Sigma$ be given. Now turn to
  constructing a coupling between two discrete processes $Y^L$ and
  $\tilde Y^L$ with respective initial states $Y^L_0$ and
  $\tilde Y^L_0$. Let $s^1, s^2, \tilde s^1, \tilde s^2$ be four
  independent Markov jump processes with the transition rates of
  figure \ref{fig:single_speed_transition_rates_slowman_2016}
  (resp. \ref{fig:single_speed_transition_rates_slowman_2017}) and
  respective initial states
  $\sigma^1_0, \sigma^2_0, \tilde \sigma^1_0, \tilde
  \sigma^2_0$. Define
  $\tau_i = \inf \{t \ge 0 : \sigma^i(t) = \tilde \sigma^i(t)\}$ and
  set
\begin{align*}
  \sigma^i(t) = s^i(t) \text{ for all } t \ge 0, \quad \quad \tilde \sigma^i(t) = \tilde s^i(t) \text{ for } t < \tau_i \text{ and } \tilde \sigma^i(t) = s^i(t) \text{ for } t \ge \tau_i.
\end{align*}

Let $N_1^L$ and $N_2^L$ be two independent Poisson clocks with
parameter $\gamma_L = (L-1) / \ell$ independent of
the~$\sigma^i$. Finally, set
$Y^L = \mathcal Y (\sigma^1, \sigma^2, N_1^L, N_2^L, y^L_0, L)$ and
$\tilde Y^L = \mathcal Y (\tilde \sigma^1, \tilde \sigma^2, N_1^L, N_2^L,
\tilde y^L_0, L)$.
\end{definition}

\begin{rem}
With this coupling,
\begin{align*}
\sigma^i(t_0) = \tilde \sigma^i(t_0) &\implies \sigma^i(t) = \tilde \sigma^i(t) \text{ for all } t \ge t_0, \quad \quad
Y^L(t_0) = \tilde Y^L(t_0) \implies Y^L(t) = \tilde Y^L(t) \text{ for all } t \ge t_0.
\end{align*}
\end{rem}
First, recall the coupling
characterization of the total variation distance
$$
\sup_{\mu_L} \tv{\mu_L P^L_t - \pi_L} \le \sup_{Y_0^L, \tilde Y_0^L} \mathbb P_{Y_0^L, \tilde Y_0^L} \left( Y^L(t) \neq \tilde Y^L(t) \right).
$$
Therefore, it is left to determine the stopping time
$\inf \{ t \ge 0 : Y^L(t) = \tilde Y^L(t) \}$.  Understanding the
first time where both velocities are identical
$\tau_\sigma = \max_i \tau_i$ is straightforward because it does not
depend on $L$. Thus the main obstacle is to determine the time that
elapses between $\tau_\sigma$ and
$\tau_y = \inf\{ t \ge \tau_\sigma : y^L(t) = \tilde y^L(t) \}$. The
following deterministic lemma enables this by linking
$\tau_y - \tau_\sigma$ to the variations of $S^L(t)$. \green{Here, and in the rest of this section, $S^L(t)$ and $I(t)$ are defined as in
	lemma \ref {lem:toy_problem}.}

\begin{lem} \label{lem:deterministic_discrete_coupling_lemma} Let
  $Y^L$ and $\tilde Y^L$ be as in definition
  \ref{def:discrete_discrete_coupling} and assume that
  $\sigma^1_0 = \tilde \sigma^1_0$ and
  $\sigma^2_0 = \tilde \sigma^2_0$. For any fixed realization, one has
$$
\left| S^L(T) \right| \ge L - 1 \implies \inf \{ t \ge 0 : y^L(t) = \tilde y^L(t) \} \le T,
$$
\green{where $S^L(t)$ is defined as in lemma~\ref{lem:toy_problem}.}
\end{lem}

\begin{proof} Assume without loss of generality that
  $y^L(0) \le \tilde y^L(0)$ and $S^L(T) \ge L - 1$. Because
  $\sigma^1 = \tilde \sigma^1$ and $\sigma^2 = \tilde \sigma^2$ one
  has that $y^L(0) \le \tilde y^L(0)$ implies that
  $y^L(t) \le \tilde y^L(t)$ for all $t \ge 0$.

  Assume by contradiction that $y^L(t) \neq \tilde y^L(t)$ for all $t \le T$. If $y^L(t) = L$ for some $t \le T$ then $L = y^L(t) \le \tilde y^L(t) \le L$ which would contradict $y^L(t) \neq \tilde y^L(t)$ for all $t \le T$. Hence $y^L(t) < L$ for all $t \le T$.
  
  It can be shown by induction that $y^L(t) < L$ for all $t \le T$
  implies $y^L(t) \ge y^L(0) + S^L(t)$ for all $t \le T$. Hence
  $y^L(T) \ge y^L(0) + S^L(T) \ge L$ which is a contradiction.
\end{proof}

The following lemma establishes a crude quantitative result on the
variations of $I(t)$ and then uses the convergence of
$\frac{1}{\gamma_L} S^L(t)$ to $I(t)$ to deduce a result on the
variations of $S^L(t)$ for large $L$.

\begin{lem} \label{lem:variation_of_discrete_unbounded_process}

\begin{itemize}
\item[(i)] For all $\delta > 0$, one has
  $\liminf_{t \rightarrow +\infty} \inf_{\sigma} \mathbb P_{\sigma}
  \left( \left| I(t) \right| \ge \delta \right) \ge 1/3$.
\item[(ii)] There exist $T > 0$ and $L_0 \in \mathbb N$ such that, for
  all $L \ge L_0$, one has
  $\inf_{\sigma \in \Sigma} \mathbb P_{\sigma} \left( \left| S^L(T)
    \right| \ge L - 1 \right) \ge 1/9$.
\end{itemize}

\end{lem}

\begin{proof} (i) For the sake of conciseness, only consider the
  case of the DITP. The same arguments apply up to some slight
  modifications to the DFTP.
  
  Set $I_k(t) = \int_0^t \sigma^k(s) ds$ for $k = 1, 2$. By
  Feynman-Kac formulae,
$$
\begin{pmatrix}
\mathbb E_{-1} \left[ e^{\zeta I_k(t)} \right] \\ \mathbb E_{1} \left[ e^{\zeta I_k(t)} \right] \\
\end{pmatrix} = e^{t(\mathfrak Q + \zeta V)}\begin{pmatrix}
1 \\ 1
\end{pmatrix} \text{ with } \mathfrak Q = \begin{pmatrix}
-\omega & \omega \\ \omega & -\omega
\end{pmatrix}
\text{ and }
V = \begin{pmatrix}
-1 & 0 \\ 0 & 1
\end{pmatrix}.
$$
Hence one gets
$$\mathbb E_{\pm1} \left[ e^{\zeta I_k(t)} \right] = e^{-t \omega} \left(
  \frac{(\omega \pm\zeta ) \sinh \left(t \sqrt{\zeta ^2+\omega
        ^2}\right)}{\sqrt{\zeta ^2+\omega ^2}}+\cosh \left(t
    \sqrt{\zeta^2+\omega ^2}\right) \right).$$ For $k = 1, 2$ and
$\sigma^k_0 \in \{-1, 1\}$, differentiating w.r.t.~\( \zeta \) at \( \zeta = 0 \) and taking
$t \rightarrow +\infty$ yields
\begin{align*}
\lim_{t \rightarrow +\infty} \frac{\mathbb{E}_{\sigma^k_0} \left[ I_k(t) \right]}{t^{1/2}} &= 0, &  \lim_{t \rightarrow +\infty} \frac{\mathbb{E}_{\sigma^k_0} \left[ I_k(t)^2 \right]}{t} &= \frac{1}{\omega},\\
\lim_{t \rightarrow +\infty} \frac{\mathbb{E}_{\sigma^k_0} \left[ I_k(t)^3 \right]}{t^{3/2}} &= 0, & \lim_{t \rightarrow +\infty} \frac{\mathbb{E}_{\sigma^k_0} \left[ I_k(t)^4 \right]}{t^2} &= \frac{3}{\omega^2}.
\end{align*}
Using $I(t) = I_2(t) - I_1(t)$ and the independence of $I_1(t)$ and
$I_2(t)$ one sees that for all $\sigma \in \Sigma$
$$
\lim_{t \rightarrow +\infty} \mathbb E_{\sigma} \left[I(t)^2\right] = +\infty \text{ and } \lim_{t \rightarrow +\infty} \frac{\mathbb E_{\sigma} \left[I(t)^2\right]^2}{\mathbb E_{\sigma} \left[I(t)^4\right]} = \frac{1}{3}.
$$
Finally, when $t$ is large enough
$\frac{\delta^2}{\mathbb E \left[I(t)^2\right]} \in (0, 1)$ so by the
Paley-Zygmund inequality
\begin{align*}
\mathbb P_{\sigma} \left( \left| I(t)\right| \ge \delta \right) = \mathbb P_{\sigma} \left( I(t)^2 \ge \frac{\delta^2}{\mathbb E_\sigma \left[I(t)^2\right]} \mathbb E_\sigma \left[I(t)^2\right] \right) \ge \left( 1 - \frac{\delta^2}{\mathbb E_\sigma \left[I(t)^2\right]} \right)^2 \frac{\mathbb E_\sigma \left[I(t)^2\right]^2}{\mathbb E_\sigma \left[I(t)^4\right]}
\end{align*}
and the result follows.

(ii) Assertion (i) implies that there exists $T > 0$ such that
$ \inf_{\sigma \in \Sigma} \mathbb{P}_{\sigma} \left( \left| I(T)
  \right| \ge \ell + 1\right) \ge \frac{2}{9} $ and by proposition
\ref{lem:toy_problem} there exists $L_0 \in \mathbb N$ such that for
all $L \ge L_0$
$$
\sup_{\sigma \in \Sigma} \mathbb P_{\sigma} \left(  \left| \frac{\ell}{L - 1} S^L(T) - I(T) \right| \ge 1 \right) \le \frac{1}{9}.
$$
Hence for $L \ge L_0$ and $\sigma \in \Sigma$
\begin{align*}
  \mathbb P_\sigma \left( \left| S^L(t) \right| \ge L - 1 \right) \ge \mathbb P_\sigma \left( \left| I(T) \right| \ge \ell+ 1 \right) - \mathbb P_\sigma \left( \left| \frac{\ell}{L - 1} S^L(T) - I(T)\right| \ge 1 \right) \ge \frac{2}{9} - \frac{1}{9}.
\end{align*}
\end{proof}

The key uniformity result can now be proven.

\begin{proof}[Proposition~\ref{prop:uniform_discrete_mixing}]
  For all $L \ge 2$ there exist $C_L, \rho_L > 0$ such that
  $\sup_{\mu_L} \left\| \mu_L P^L_t - \pi_L \right\|_\text{TV} \le C_L
  e^{-\rho_L t}$ so it is enough to show that there exists
  $L_0 \in \mathbb N$ such that
  $\sup_{L \ge L_0} \sup_{\mu_L} \left\| \mu_L P^L_t - \pi_L
  \right\|_\text{TV} \le C e^{-\rho t}$ for some $C, \rho>0$.
  
  Lemma \ref{lem:variation_of_discrete_unbounded_process} guarantees
  that there exist $T > 0$ and $L_0 \in \mathbb N$ such that for all
  $L \ge L_0$
$$
\inf_{\sigma \in \Sigma} \mathbb P_{\sigma} \left( \left| S^L(T) \right| \ge L - 1 \right) \ge \frac{1}{9}.
$$

Recall
$\tau_i = \inf \{ t \ge 0 : \sigma^i(t) = \tilde \sigma^i(t) \}$,
$\tau_\sigma = \max_{i = 1, 2} \tau_i$ and
$\tau_y = \inf\{t \ge \tau_\sigma: y^L(t) = \tilde y^L(t)\}$. Set
$$
p^* = \inf_{\sigma^1_0, \tilde \sigma^1_0} \mathbb P_{(\sigma^1_0, \tilde \sigma^1_0)} \left( \sigma^1(1) = \tilde \sigma^1(1)\right).
$$

For all $k \in \mathbb N$ and $L \ge L_0$
\begin{align*}
\sup_{\mu_L} \left\| \mu_L P^L_{(T + 1)k} - \pi_L \right\|_\text{TV} &\le  \sup_{Y^L_0, \tilde Y^L_0} \mathbb P_{Y^L_0, \tilde Y^L_0} \left( Y((T + 1)k) \neq \tilde Y^L((T + 1)k) \right), \\
&\le  \sup_{Y^L_0, \tilde Y^L_0} \mathbb P_{Y^L_0, \tilde Y^L_0} \left(  \tau_y > (T + 1)k\right), \\
&\le \sup_{Y^L_0, \tilde Y^L_0} \left[\mathbb P_{Y^L_0, \tilde Y^L_0} \left( \tau_y - \tau_\sigma > k T \right) + \mathbb P_{Y^L_0, \tilde Y^L_0} \left( \tau_1 > k \right) + \mathbb P_{Y^L_0, \tilde Y^L_0} \left( \tau_2 > k \right)\right], \\
&\le \left( 8/9 \right)^k + (1 - p^*)^k + (1 - p^*)^k,
\end{align*}
where  lemma
\ref{lem:deterministic_discrete_coupling_lemma} and
\ref{lem:variation_of_discrete_unbounded_process} are iterated to bound
$\mathbb P \left( \tau_y - \tau_\sigma > k T \right)$ by $(8/9)^k$.

This concludes the proof since
$t \mapsto \sup_{L \ge L_0} \sup_{\mu_L} \left\| \mu_L P^L_{t} - \pi_L
\right\|_\text{TV}$ is decreasing.
\end{proof}

\section{Mixing times} \label{sec:mixing_time_of_the_continuous_process}

In this section, the dependence of the mixing time
$$
t_\mix(\epsilon) = \inf \left\{ t \ge 0 : \sup_{\mu} \, \tv{\mu P_t - \pi} \le \epsilon \right\}
$$
of the continuous process on the model parameters $\omega$
(resp. $\alpha$ and $\beta$) and $\ell$ is determined up to a multiplicative
constant. Denote by $t_\mix^\IT(\epsilon)$
(resp. $t_\mix^\FTD(\epsilon)$) the mixing time of the CITP
(resp. CFTP).

\begin{thm}[Mixing time] \label{thm:mixing_time} \begin{itemize}
  \item[(i)] For all $\epsilon \in (0, 1)$ there exists
    $C(\epsilon), C'(\epsilon) > 0$ such that
$$
C(\epsilon) \frac{1}{\omega} \left(1 + \omega^2 \ell^2\right) \le t_\mix^\IT(\epsilon) \le {C'(\epsilon)}\frac{1}{\omega} \left(1 + \omega^2 \ell^2\right).
$$
\item[(ii)] For all $\epsilon \in (0, 1)$ there exists $C(\epsilon), C'(\epsilon) > 0$ such that
$$
C(\epsilon) \left( \frac{1}{\alpha} + \frac{1}{\beta} \right) \left(1 + \alpha^2 \ell^2 \right) \le t_\mix^\FTD(\epsilon) \le {C'(\epsilon)} \left( \frac{1}{\alpha} + \frac{1}{\beta} \right) \left(1 + \alpha^2 \ell^2\right).
$$
\end{itemize}
\end{thm}
By standard arguments, e.g.~\cite{peres09}, one has that
$C'(\epsilon)\sim \log(1/\epsilon)$. However the expression for
$C(\epsilon)$ is more involved. Notice that matching lower
and upper bounds are obtained meaning that the dependence on the model parameters
is optimal. 
The remainder of this section is dedicated to the proof of theorem
\ref{thm:mixing_time}. The lower and upper bound of
assertions (i) and (ii) are proven separately.

The exact forms of the invariant measure of the CITP and the CFTP are
used during the proofs of the theorems, noticeably the fact that they
exhibit Dirac masses at $0$ and $\ell$. While derived in the discrete
case in \cite{slowman16,slowman17}, universality classes have been
obtained in \cite{hahn23} characterizing the exact form of the
invariant measure in the continuous case in a general setting. For the
sake of completeness, the invariant measure for the inter-particle
separation for the CITP is explicited in Proposition
\ref{prop_invariant_measure_slowman_2016} and for the CFTP in
Proposition \ref{prop_invariant_measure_slowman_2017}.

\subsection{Lower bound}

The lower bounds in assertions (i) and (ii) of theorem
\ref{thm:mixing_time} can be shown using identical arguments based on
the identification of the slow observables. For the sake of conciseness, only the lower
bound of assertion (ii) is proven. In
the following lemma, a first lower bound is deduced by looking at the
velocities.

\begin{lem} \label{lem_mixing_constant_lower_bound_slowman_2017}
	For all $\epsilon \in (0, 1)$ one has
	$$
	\frac{|\log(\epsilon)|  }{2\min \left(\alpha, \beta\right)} \le t_\mix^\FTD(\epsilon).
	$$
\end{lem}

\begin{proof}
  If $t \ge 0$ is such that
  $\sup_{\mu} \, \tv{\mu P_{t} - \pi} \le \epsilon$, then, for any
  $X_0 = (x, (\sigma^1_0, \sigma^2_0))\in ]0,\ell[\times\Sigma$,
$$
\mathbb P_{X_0} \left(X(t) = X_0\right) - \pi(X_0) = \mathbb P_{X_0} \left(X(t) = X_0\right) \le \epsilon,
$$
as there is no mass in any such state $X_0$ in the explicit formula for
$\pi$ (see Proposition
\ref{prop_invariant_measure_slowman_2017}). In particular, for $\sigma^1_0=\sigma^2_0$,
$$
\epsilon\geq \sup_{\sigma^1_0}\mathbb P_{X_0} \left(X(t) = X_0 \right) \ge  \sup_{\sigma^1_0}\mathbb P_{X_0} \left(\sigma^1(s) = \sigma^2(s) \text{ for all } s \le t \right) \geq e^{-2 \min(\alpha,\beta) t}. $$
Thus, $$ t_\mix^\FTD (\epsilon) \ge \frac{|\log(\epsilon)|}{2\min(\alpha,\beta)}
$$ and the
result follows.

\end{proof}

Exploring a positive fraction of the state space requires a change in
position of order $\ell$. The two following lemmas confirm that, in the
diffusive regime, this takes order $\ell^2$ time
leading to a lower bound with quadratic dependence on $\ell$.

\begin{lem} \label{lem_mixing_lower_bound_concentration} Let
  $\sigma^1$ and $\sigma^2$ be two independent Markov processes with
  the transition rates of figure
  \ref{fig:single_speed_transition_rates_slowman_2017}.
	
  For all $\epsilon > 0$ there exists $u(\epsilon), A(\epsilon) > 0$
  such that for all $\delta > 0$, choosing
  $t^* = u(\epsilon) \left(1 + \frac{\alpha}{\beta}\right) \alpha
  \delta^2$,
	$$
	\alpha \delta \ge A(\epsilon) \implies \max_{\sigma \in
          \Sigma} \mathbb{P}_{\sigma} \left( \sup_{0 \le t \le t^*}
          \left| \int_0^t (\sigma^2(s) - \sigma^1(s)) ds \right|>
          \delta \right) \le \epsilon.
	$$
\end{lem}

\begin{proof} One has
	\begin{align*}
          \mathbb{P}_{\sigma} \left( \sup_{0 \le t \le t^*} \left| \int_0^t (\sigma^2(s) - \sigma^1(s)) ds \right| > \delta \right) \le 2\, \mathbb{P}_{\sigma^1_0} \left( \sup_{0 \le t \le t^*} \left| \int_0^t \sigma^1(s) ds \right| > \delta / 2\right)
	\end{align*}
	so that it suffices to prove that for all $\epsilon \in (0,
        1)$ there exist $u(\epsilon), A(\epsilon) > 0$ such that for
        $t^* = u(\epsilon) \left(1 + \frac{\alpha}{\beta}\right) \alpha
        \delta^2$
	\begin{align*}
	\alpha\delta \ge A(\epsilon) \implies \max_{\sigma^1_0 = -1, 0, 1} \, \mathbb{P}_{\sigma^1_0} \left( \sup_{0 \le t \le t^*} \left| \int_0^t \sigma^1(s) ds \right| > \delta\right) \le \epsilon.
	\end{align*}
	Define $\tilde \tau_{-1} = 0$ as well as
        $\tau_n = \inf \{ t \ge \tilde \tau_{n-1} : \sigma^1(t) = 0
        \}$ and
        $\tilde \tau_n = \inf \{ t \ge \tau_{n} : \sigma^1(t) \neq
        0\}$ for $n \ge 0$. Denote
        $D_i = \int_{\tau_{i-1}}^{\tau_i} \sigma^1(s) ds$ and observe
        that
	$$
	\sup_{t \in [0, \tau_n]} \left| \int_0^t \sigma^1(s) ds
        \right| = \max_{0 \le k \le n} \left| \int_0^{\tau_k}
          \sigma^1(s) ds \right| \le \left| \int_0^{\tau_0}
          \sigma^1(s) ds \right| + \max_{1 \le k \le n} \left| \sum_{i
            = 1}^k D_i \right|.
	$$	
	Set $n = \lceil 2 u \alpha^2 \delta^2 \rceil$ and $t^* = u \left(1 + \frac{\alpha}{\beta}\right)\alpha\delta^2$. One has
	\begin{align*}
          \mathbb{P}_{\sigma^1_0} \left( \sup_{0 \le t \le t^*} \left| \int_0^t \sigma^1(s) ds \right|> \delta \right) &\le \mathbb{P}_{\sigma^1_0} \left( \left| \int_0^{\tau_0} \sigma^1(s) ds \right| > \delta/2 \text{ or } \max_{1 \le k \le n} \left| \sum_{i = 1}^k D_i \right| > \delta/2 \text{ or } t^* > \tau_n \right) \\
                                                                                                                       &\le \underbrace{\mathbb{P}_{\sigma^1_0} \left( \left| \int_0^{\tau_0} \sigma^1(s) ds \right| > \delta/2\right)}_{\text{(I)}} + \underbrace{\mathbb{P}_{\sigma^1_0} \left(\max_{1 \le k \le n} \left| \sum_{i = 1}^k D_i \right| > \delta/2\right)}_{\text{(II)}} \\
                                                                                                                       &\qquad\qquad+ \underbrace{\mathbb{P}_{\sigma^1_0} \left( t^* > \tau_n \right)}_{\text{(III)}}.
	\end{align*}
	
	It remains to show that there exist $u = u(\epsilon) > 0$ and
        $A = A(\epsilon) > 0$ such that each of the terms (I)--(III)
        is bounded by $\epsilon /3$ if $\alpha \delta \ge A$.
	
	\underline{Bounding (I).} If $\sigma^1_0 = 0$ then (I) is zero. If $\sigma^1_0 = \pm 1$ then Markov's inequality yields
	$$
	\mathbb{P}_{\sigma^1_0} \left( \left| \int_0^{\tau_0} \sigma^1(s) ds \right| > \delta/2\right) \le \frac{\mathbb{E}_{\sigma^1_0} \left| \int_0^{\tau_0} \sigma^1(s) ds \right|}{\delta/2} = \frac{2}{\alpha \delta}
	$$
	so choosing $A \ge 6/\epsilon$ leads to the desired bound.
	
	\underline{Bounding (II).} \green{Notice that for $\sigma^1_0 = -1, 0, 1$, under the probability $\mathbb P_{\sigma^1_0}$, the random variables $(D_i)_{i \ge 1}$ are
        i.i.d.~random variables with density~$\frac12 \alpha e^{-\alpha|x|} dx$. In particular $\mathbb E_{\sigma^1_0}[D_i] = 0$ and $\mathbb E_{\sigma^1_0}[D_i^2] = 2/\alpha^2$. Kolmogorov's inequality yields}
	\begin{align*}
	\mathbb{P}_{\sigma^1_0} \left( \max_{1 \le k \le n} \left| \sum_{i = 1}^k D_i \right|  > \delta/2\right) \le \frac{\mathbb{E}_{\sigma^1_0}\left( \sum_{i = 1}^n D_i \right)^2}{\delta^2/4} = \frac{4 n \mathbb{E}_{\sigma^1_0} \left[ D_1^2 \right]}{\delta^2} = 8 \frac{\lceil 2 u \alpha^2 \delta^2\rceil}{\alpha^2 \delta^2}.
	\end{align*}
	
	If $2 u \alpha^2 \delta^2 \ge 1$ one has
        $\lceil 2 u \alpha^2 \delta^2 \rceil \le 4 u \alpha^2 \delta^2
        $ so that
	$$
	\mathbb{P}_{\sigma^1_0} \left( \max_{1 \le k \le n} \left| \sum_{i = 1}^k D_i \right|  > \delta/2\right) \le 32 u
	$$
	and the desired bound holds if $u \le \epsilon / 96$ and
        $A \ge 1/\sqrt{2u}$.
	
	\underline{Bounding (III).} Notice that the random variables
        $(\tau_i - \tau_{i-1})_{i\ge 1}$ are i.i.d. and have mean
        $1/\alpha + 2/\beta$ and variance $1/\alpha^2 +
        4/\beta^2$. Chebyshev's inequality gives us
	\begin{align*}
	\mathbb{P}_{\sigma^1_0} \left( t^* > \tau_n \right) &\le \mathbb{P}_{\sigma^1_0} \left( t^* > \sum_{i = 1}^n (\tau_i - \tau_{i - 1}) \right) = \mathbb{P}_{\sigma^1_0} \left( t^* - n (1/\alpha+ 2/\beta) > \sum_{i = 1}^n (\tau_i - \tau_{i - 1}) - n (1/\alpha+ 2/\beta)\right), \\
	&\le \frac{n \text{Var}\left(\tau_1 - \tau_0\right)}{(t^* - n (1/\alpha + 2/\beta))^2} \le \frac{\lceil 2 u \alpha^2\delta^2\rceil}{\alpha^4u^2\delta^4} \frac{ 4\alpha^2 + \beta^2}{(\alpha + \beta)^2}.
	\end{align*}
	
	If $2 u \alpha^2 \delta^2 \ge 1$ one has
        $\lceil 2 u \alpha^2 \delta^2 \rceil \le 4 u \alpha^2 \delta^2
        $ so that
	$$
	\mathbb{P}_{\sigma^1} \left( t^*(u) > \tau_n \right) \le \left( \sup_{\alpha, \beta > 0}\frac{4\alpha^2 + \beta^2}{(\alpha + \beta)^2} \right) \frac{4}{u \alpha^2 \delta^2} = \frac{16}{u \alpha^2 \delta^2}
	$$
	hence the desired bound for (III) holds if
        $A \ge \max \left( 1/\sqrt{2 u}, \sqrt{\frac{48}{\epsilon u}}
        \right)$.
\end{proof}

\begin{lem} \label{lem_mixing_quadratic_lower_bound_slowman_2017} For
  all $\epsilon \in (0, 1)$ there exist
  $\tilde A (\epsilon), \tilde u (\epsilon) > 0$ such that
	$$
	\alpha \ell \ge \tilde A (\epsilon) \implies t_\mix^\FTD (\epsilon) \ge \tilde u(\epsilon) (1/\alpha + 1/\beta) \alpha^2 \ell^2.
	$$
\end{lem}

\begin{proof} Set $M = ((1 - a) \ell/2, (1 + a) \ell/2) \times \Sigma$
  where $a \in (0, 1)$ is fixed later. The explicit formula for
  the invariant measure $\pi$ implies $\pi(M^c) \ge 1 - a$. It is
  clear from definition
  \ref{def:continuous_process_min_max_construction} that
	$$
	\mathbb{P}_{(\ell/2, 1, 1)} \left( X(t) \in M^c \right) \le
        \mathbb{P}_{(1, 1)} \left( \sup_{0 \le \theta \le t} \left|
            \int_0^\theta (\sigma^2(s) - \sigma^1(s)) ds \right| \ge
          \frac{a \ell}{2}\right).
	$$
	
	By lemma \ref{lem_mixing_lower_bound_concentration}, with
        $t^* = u\left(\frac{1 - \epsilon}{4}\right) (1 +
        \frac{\alpha}{\beta}) \alpha \delta^2$,
	$$
	\alpha \delta \ge A \left( \frac{1 - \epsilon}{4}\right) \implies \mathbb{P}_{(1, 1)} \left( \sup_{0 \le t \le t^*} \left| \int_0^t (\sigma^2(s) - \sigma^1(s)) ds \right| \ge \delta \right) \le \frac{1 - \epsilon}{4}.
	$$
	
	Taking $\delta = a \ell / 2$ one gets with
        $t^* = \frac{a^2}{4} u \left( \frac{1 - \epsilon}{4}\right)
        (1+ \frac{\alpha}{\beta}) \alpha \ell^2$,
	\begin{align*}
	\alpha \ell \ge \frac{2}{a} A \left(\frac{1 - \epsilon}{4}\right) \implies \mathbb{P}_{(1, 1)} \left( \sup_{0 \le t \le t^*} \left| \int_0^t (\sigma^2(s) - \sigma^1(s)) ds \right| \ge \frac{a \ell}{2} \right) \le \frac{1 - \epsilon}{4}.
	\end{align*}
	
	Set $a = \frac{1 - \epsilon}{2}$,
        $\tilde A(\epsilon) = \frac{2}{a} A \left(\frac{1 -
            \epsilon}{4}\right)$,
        $\tilde u (\epsilon) = \frac{a^2}{4} u \left( \frac{1 -
            \epsilon}{4} \right)$ and
        $t^* = \tilde u (\epsilon) (1 + \frac{\alpha}{\beta}) \alpha
        \ell^2$. If $\alpha \ell \ge \tilde A (\epsilon)$ and
        $t^* \ge t_\mix^\IT(\epsilon)$ then
	\begin{align*}
	1 - \frac{1 - \epsilon}{2} - \frac{1 - \epsilon}{4}& \le 1 - a - \mathbb{P}_{(1, 1)} \left( \sup_{0 \le \theta \le t^*} \left| I(\theta) \right| \ge \frac{a \ell}{2}\right)\\
	& \le \pi (M^c) - \mathbb{P}_{(\ell/2, 1, 1)} (X(t^*) \in M^c) \le \sup_{\mu} \tv{\mu P_{t^*} - \pi} \le \epsilon
	\end{align*}
	which is a contradiction because $\epsilon \in (0, 1)$. Hence
        $\alpha \ell \ge \tilde A (\epsilon)$ implies
        $t_\mix^\IT(\epsilon) \ge \tilde u (\epsilon) (1 +
        \frac{\alpha}{\beta}) \alpha \ell^2$.
\end{proof}

Assertion (ii) of theorem \ref{thm:mixing_time} now follows from lemma
\ref{lem_mixing_constant_lower_bound_slowman_2017} and lemma
\ref{lem_mixing_quadratic_lower_bound_slowman_2017}.

\subsection{Upper bound for the instantaneous tumble process}

We show the upper bound of theorem \ref{thm:mixing_time} by coupling
techniques, first for assertion (i) in this section and for assertion
(ii) in the next. We start by defining the continuous analog of the
discrete coupling of definition \ref{def:discrete_discrete_coupling}.

\begin{definition}[Continuous-continuous
  coupling] \label{def:continuous_continuous_coupling} Let
  $(x_0, (\sigma^1_0, \sigma^2_0)), (\tilde x_0,(\tilde \sigma^1_0,
  \tilde \sigma^2_0)) \in [0, \ell] \times \Sigma$ be given. Let
  $s^1, s^2, \tilde s^1, \tilde s^2$ be four independent Markov jump
  processes with the transition rates of figure
  \ref{fig:single_speed_transition_rates_slowman_2016} (resp. figure
  \ref{fig:single_speed_transition_rates_slowman_2017}) and respective
  initial states
  $\sigma^1_0, \sigma^2_0, \tilde \sigma^1_0, \tilde \sigma^2_0$. Set
  $\tau_i = \inf \{t \ge 0: s_i(t) = \tilde s_i(t)\}$ and define
\begin{align*}
  \sigma^i(t) = s^i(t) \text{ for } t \ge 0, \quad \quad \tilde \sigma^i(t) = \tilde s^i(t) \text{ for } t \le \tau_i \text{ and } \tilde \sigma^i(t) = s^i(t) \text{ for } t\ge \tau_i.
\end{align*}

Finally set $X = \mathcal X(\sigma^1, \sigma^2, x_0, \ell)$ and
$\tilde X = \mathcal X (\tilde \sigma^1, \tilde \sigma^2, \tilde x_0,
\ell)$.
\end{definition}
The goal for the construction of this coupling is that
\begin{align*}
\sigma^i(t_0) = \tilde \sigma^i(t_0) &\implies \sigma^i(t) = \tilde \sigma^i(t) \text{ for all } t \ge t_0, \\
X(t_0) = \tilde X(t_0) &\implies X(t) = \tilde X(t) \text{ for all } t \ge t_0.
\end{align*}

Similarly to the coupling of the discrete process, the challenge lies
in the determination of the time between the velocity coupling
$\tau_\sigma = \max_i \tau_i$ and the one of the positions. The key
ingredient here is that, once the velocities are identical, the
interdistance $|x(t)-\tilde x(t)|$ is nonincreasing and the order
between the positions is preserved, as, 
\begin{align*}
x(\tau_\sigma) \le \tilde x(\tau_\sigma) \implies x(t) \le \tilde x(t) \text{ for all } t \ge \tau_\sigma, \\
x(\tau_\sigma) \ge \tilde x(\tau_\sigma) \implies x(t) \ge \tilde x(t) \text{ for all } t \ge \tau_\sigma.
\end{align*}
Furthermore, the position coupling happens with $\sigma^1\sigma^2=-1$
either in $0$ with $\sigma^1=1$ or $l$ with
$\sigma^1=-1$.
Therefore, we deduce,
$$
\inf \{ t \ge 0 : X(t) = \tilde X(t) \}= \inf\left\{ t \ge
  \tau_\sigma : x(t) \vee \tilde x(t) = 0 \right\} \wedge \inf \left\{t \ge \tau_\sigma: x(t) \wedge \tilde x(t) = l \right\} =: \tau_\text{\normalfont
  coupling}.
$$
Thus upper bounding the mixing times reduces to upper bounding a hitting time.
Section~\ref{secdavis} and proposition \ref{prop_min_max_form} allows us to
explicitly compute hitting times by solving systems of differential equations
using the theory laid out in section 3 of \cite{davis93}.

\begin{lem} \label{lem_maximum_hitting_time_of_0}
Let $\tau = \inf \left\{ t \ge 0 : X_D(t) = (0, *^0_{(1, -1)} ) \right\}$. There exists $C > 0$ such that
\begin{align*}
\sup_{X_0 \in E_D} E_{X_0} \left[\tau\right] \le C \frac{1}{\omega} \left(1 + \omega^2 \ell^2\right).
\end{align*}
\end{lem}

\begin{proof} \textcolor{myGreen}{By the characterization of the generator provided in lemma~\ref{prop:generator_continuous_process} (see Section 3 of~\cite{davis93})}, the mean hitting time
  $f(x, \sigma) = E_{(x, \sigma)} \left[\tau\right]$ satisfies the
  differential-algebraic system of equations
$$
\left.
\begin{array}{r}
-2\omega f(x, 1, 1) + \omega f(x, 1, -1) + \omega f(x, -1, 1) + 1 = 0\\
-2 \partial_x f(x, 1, -1) - 2\omega f(x, 1, -1) + \omega f(x, 1, 1) + \omega f(x, -1, -1) + 1 = 0\\
2 \partial_x f(x, -1, 1) -2\omega f(x, -1, 1) + \omega f(x, 1, 1) + \omega f(x, -1, -1) + 1 = 0\\
-2\omega f(x, -1, -1) + \omega f(x, 1, -1) + \omega f(x, -1, 1) + 1 = 0
\end{array}
\right\}
\text{ for all } x \in (0, \ell)
$$
with boundary conditions $f(0, 1, -1) = 0$ and
$f(\ell, -1, 1) = f(\ell, 1, -1) + 2/\omega$ yielding
\begin{align*}
\mathbb{E}_{(x, (1, 1))}[\tau] = \frac{2   { (\ell + x )} \omega + { (2   \ell x - x^{2} )} \omega^{2} + 3   }{2 \omega},&\qquad\mathbb{E}_{(x, (1, -1))}[\tau] = \frac{4   x  + { (2   \ell x - x^{2} )} \omega}{2}, \\
\mathbb{E}_{(x, (-1, 1))}[\tau] = \frac{4   \ell \omega + { (2   \ell x - x^{2} )} \omega^{2} + 4}{2 \omega},\qquad& \qquad\mathbb{E}_{(x, (-1, -1))}[\tau] = \frac{2   { (\ell + x )} \omega + { (2   \ell x - x^{2} )} \omega^{2} + 3}{2 \omega}.
\end{align*}

The result follows by optimizing over $x$.
\end{proof}

An upper bound for the mixing time of the continuous instantaneous tumble
process can now be given.

\begin{proof}[Upper bound of theorem~\ref{thm:mixing_time} (i)]
	
  Let
  $(x_0, (\sigma^1_0, \sigma^2_0)), (\tilde x_0, (\tilde \sigma^1_0,
  \tilde \sigma^2_0)) \in [0, \ell] \times \Sigma$ be arbitrary but
  fixed. Let $X, \tilde X$ and $\tau_1, \tau_2$ be as in definition~\ref{def:continuous_continuous_coupling}.

Set $\tau_0 = \inf \{ t \ge \tau_\sigma : x(t) = 0 \}$ and $\tilde \tau_0 = \inf \{ t \ge \tau_\sigma: \tilde x(t) = 0 \}$. Using lemma \ref{lem_maximum_hitting_time_of_0} one gets
\begin{align*}
\mathbb E_{X_0, \tilde X_0} \left[ \tau_\text{coupling} \right] &\le \mathbb E_{X_0, \tilde X_0} \left[ \max\left(\tau_0, \tilde \tau_0 \right) - \max\left(\tau_1, \tau_2 \right) \right] + \mathbb E_{X_0, \tilde X_0} \left[ \max\left(\tau_1, \tau_2 \right) \right] \\
&\le \mathbb E_{X_0, \tilde X_0} \left[\tau_0 - \max\left(\tau_1, \tau_2 \right) \right] + \mathbb E_{X_0, \tilde X_0} \left[\tilde \tau_0 - \max\left(\tau_1, \tau_2 \right) \right] + \mathbb E_{X_0, \tilde X_0} \left[\tau_1\right] +  \mathbb E_{X_0, \tilde X_0} \left[\tau_2\right]\\
&\le C \frac{1}{\omega} \left(1 + \omega^2\ell^2\right) + C \frac{1}{\omega} \left(1 + \omega^2\ell^2\right) + \frac{1}{2 \omega} + \frac{1}{2 \omega} \\
&= (2C + 1) \frac{1}{\omega} \left(1 + \omega^2 \ell^2 \right).
\end{align*}
Let $\epsilon>0$ be arbitrary but fixed and set
$t^* = (2C + 1) (1 + \omega^2 \ell^2) / (\epsilon \omega)$. One has
\begin{align*}
\sup_\mu \tv{\mu P_{t^*} - \pi } 
 \le \sup_{X_0, \tilde X_0} \mathbb P_{X_0, \tilde X_0} \left(\tau_\text{coupling} > t^* \right) \le \sup_{X_0, \tilde X_0} \frac{\mathbb E_{X_0, \tilde X_0} \left[\tau_\text{coupling}\right]}{t^*} \le \epsilon.
\end{align*}
\end{proof}

\subsection{Upper bound for the finite tumble process}

We now derive an upper bound for the mixing time of the continuous
finite tumble process using the same coupling. We start by controlling
$\tau_\sigma = \max_i \tau_i$.

\begin{lem} \label{lem_mean_speed_coupling_time_slowman_2017} Let
  $\sigma = (\sigma^1, \sigma^2), \tilde \sigma = (\tilde \sigma^1,
  \tilde \sigma^2)$ and $\tau_1, \tau_2$ be as in definition
  \ref{def:continuous_continuous_coupling}. One has
  $$ \sup_{\sigma_0, \tilde \sigma_0 \in \Sigma} \mathbb{E}_{\sigma_0,
    \tilde \sigma_0} \left[ \max \left(\tau_1, \tau_2 \right) \right]
  \leq \frac{3}{\alpha}.
$$
\end{lem}

\begin{proof}
  One has
  $\sup_{\sigma_0, \tilde \sigma_0} \, \mathbb{E}_{\sigma_0, \tilde
    \sigma_0} \left[ \max \left(\tau_1, \tau_2 \right) \right] \le
  \sup_{\sigma_0, \tilde \sigma_0} \, \mathbb E[\tau_1 + \tau_2] \le 2
  \sup_{\sigma^1_0, \tilde \sigma^1_0} \, \mathbb E_{\sigma^1_0,
    \tilde \sigma^1_0} [\tau_1]$. Hence, it is enough to show
  $\sup_{\sigma^1_0, \tilde \sigma^1_0} \mathbb E_{(\sigma^1_0, \tilde
    \sigma^1_0)}[\tau_1] \le \frac{3}{2\alpha}$.

  Set
  $k_{\sigma^1_0, \tilde \sigma^1_0} = \mathbb E_{(\sigma^1_0, \tilde
    \sigma^1_0)} \left[ \tau_1 \right]$. By first step analysis,
\begin{align*}
\sum_{(s^1, \tilde s^1)} q_{(\sigma^1_0, \tilde \sigma^1_0), (s^1, \tilde s^1)} k_{(s^1, \tilde s^1)} + 1 = 0 &\text{ if } \sigma^1_0 \neq \tilde \sigma^1_0, \\
k_{\sigma^1_0, \tilde \sigma^1_0} = 0 &\text{ if } \sigma^1_0 = \tilde \sigma^1_0,
\end{align*}
where $q$ is the discrete generator of the couple
$(\sigma^1, \tilde \sigma^1)$. Solving this system yields
\begin{align*}
\mathbb{E}_{(-1, -1)} \left[ \tau_1 \right] = \mathbb{E}_{(0, 0)} \left[ \tau_1 \right] = \mathbb{E}_{(1, 1)} \left[ \tau_1 \right] &=  0, \\
\mathbb{E}_{(0, -1)} \left[ \tau_1 \right] = \mathbb{E}_{(0, 1)} \left[ \tau_1 \right] = \mathbb{E}_{(-1, 0)} \left[ \tau_1 \right] = \mathbb{E}_{(1, 0)} \left[ \tau_1 \right] &= \frac{1}{\alpha} \frac{4 \, r + 1}{4 \, r + 2}, \\
\mathbb{E}_{(-1, 1)} \left[ \tau_1 \right] = \mathbb{E}_{(1, -1)} \left[ \tau_1 \right] &= \frac{1}{\alpha} \frac{3 \, r + 1}{2 \, r + 1},
\end{align*}
so that the result follows from optimizing over $r > 0$.
\end{proof}

We now evaluate the elapsed time between the couplings of velocity and
position. Adopting a similar strategy to that used in the discrete
case, we focus on variations of the integral process
$I(t)=\int_0^t (\sigma^2(s) - \sigma^1(s)) ds$. As a first step, we establish an
improved continuous form of the deterministic lemma
\ref{lem:deterministic_discrete_coupling_lemma} that links the
coupling time to the variations of $I(t)$.

\begin{lem} \label{lem_crucial_deterministic_lemma}
	
  Let $x(t)$ and $\tilde x (t)$ be as in definition
  \ref{def:continuous_continuous_coupling} and assume
  $(\sigma^1_0, \sigma^2_0) = (\tilde \sigma^1_0, \tilde \sigma^2_0)$.
	
  Then for all $T > 0$ and all fixed realization of
  $(\sigma^1, \sigma^2)$ one has
	$$
	\sup_{t_1 \le t_2 \le T} \left| \int_{t_1}^{t_2} (\sigma^2(s) - \sigma^1(s)) ds \right| \ge \ell \implies \inf \{t \ge 0 : x(t) = \tilde x(t) \} \le T.
	$$
\end{lem}

\begin{proof}
  Assume  without loss of generality that there exists $t_1 \le t_2 \leq T$ such that $\int_{t_1}^{t_2} (\sigma^2(s) - \sigma^1(s)) ds \ge \ell$ and that $x(0) \leq \tilde x(0)$, hence $x(t) \leq \tilde x(t)$ for all $t \geq 0$. 
  
  Assume by contradiction that $ x(t) < \tilde x(t) \leq \ell$ for all
  $t \le T$. A recursion shows that this leads to
  $x(t) \geq x(t_1) + \int_{t_1}^{t} (\sigma^2(s) - \sigma^1(s)) ds$
  for all $t \in [t_1, T]$. Thus
  $x(t_2) \geq x(t_1) + \int_{t_1}^{t_2} (\sigma^2(s) - \sigma^1(s))
  ds \geq \ell$ which is impossible.
\end{proof}

The previous deterministic result is completed by showing that
variation of order $\ell$ are reached after a time of order
$(1/\alpha + 1/\beta)(1 + \alpha^2 \ell^2)$.  Rather than considering
$\int_{t_1}^{t_2} (\sigma^2(s) - \sigma^1(s)) , ds$, let us instead
examine $\int_{\tau_0}^{\tau_n} (\sigma^2(s) - \sigma^1(s)) , ds$,
where the $\tau_i$ represent return times of $(\sigma^1, \sigma^2)$ to
a set that will be defined later. This allows us to express
$$ \int_{\tau_0}^{\tau_n} (\sigma^2(s) - \sigma^1(s)) ds = \sum_{i =
  1}^n \underbrace{\int_ {\tau_{i-1}}^{\tau_i} (\sigma^2(s) - \sigma^1(s)) ds}_{=: \D_i}$$ enabling us to leverage the fact that the $\D_i$
are i.i.d., making them more amenable to computation.

\begin{lem} \label{lem_variation_lower_bound} If $\sigma^1$ and
  $\sigma^2$ are independent Markov jump processes with the transition
  rates \ref{fig:single_speed_transition_rates_slowman_2017} then
\begin{align*}
\inf_{\sigma} \mathbb{P}_{\sigma} \left( \sup_{t_1 \le t_2 \le t^*} \left|\int_{t_1}^{t_2} (\sigma^2(s) - \sigma^1(s)) ds \right| \ge \ell \right) \ge \frac{1}{144} 
\end{align*}
where $t^* = 288  \left( \frac{1}{\alpha} + \frac{1}{\beta} \right) \left( 1 + \alpha^2 \ell^2 \right) $.
\end{lem}

\begin{proof} Define
  $\tilde{\tau}_n = \inf \{ t \ge \tau_{n-1} : \sigma^1(t) \neq
  \sigma^2(t) \}$ and
  $\tau_n = \inf \{ t \ge \tilde{\tau}_n : \sigma^1(t) = \sigma^2(t)
  \}$ for $n \geq 0$ with the convention that $\tau_{-1} = 0$.
	
  \underline{Step 1.} Solving the adequate system of linear equations,
  as for Lemma \ref{lem_maximum_hitting_time_of_0}, yields
	\begin{align*}
	\mathbb{E}_{(0, 0)}[\tau_0] &= \left( \frac{1}{\alpha} + \frac{1}{\beta} \right) \frac{2 \, r^{2} + 5 \, r + 1}{4 \, r^{2} + 6 \, r + 2},\\
	\mathbb{E}_{(1, 0)} \left[ \tau_0 \right] = \mathbb{E}_{(-1, 0)} \left[ \tau_0 \right] = \mathbb{E}_{(0, 1)} \left[ \tau_0 \right] = \mathbb{E}_{(0, -1)} \left[ \tau_0 \right] &= \left( \frac{1}{\alpha} + \frac{1}{\beta} \right) \frac{4 \, r + 1}{4 \, r^{2} + 6 \, r + 2}, \\
	\mathbb{E}_{(1, 1)} \left[ \tau_0 \right] = \mathbb{E}_{(-1, -1)} \left[ \tau_0 \right] = \mathbb{E}_{(-1, 1)} \left[ \tau_0 \right] = \mathbb{E}_{(1, -1)} \left[ \tau_0 \right] &= \left( \frac{1}{\alpha} + \frac{1}{\beta}\right) \frac{3 \, r + 1}{2 \, r^{2} + 3 \, r + 1},
	\end{align*}
	so that optimizing over $r > 0$ leads to
        $\sup_\sigma \mathbb{E}_\sigma \left[ \tau_0 \right] \le
        1/\alpha + 1/\beta$.
	
	\underline{Step 2.} First step analysis yields that for $n =
        0$ and $\sigma \in \Sigma$
	\begin{align*}
	\mathbb{P}_{\sigma} \left( \sigma(\tau_0) = (0, 0) \right) =  \frac{2 \, \alpha}{2 \, \alpha + \beta} \text{ and }
	\mathbb{P}_{\sigma} \left( \sigma(\tau_0) = (1, 1) \right) = \mathbb{P}_{\sigma} \left( \sigma(\tau_0) = (-1, -1) \right) = \frac{\beta}{2 \, \alpha + \beta}
	\end{align*}
	and a simple iteration yields the same result for $n \ge 1$
        and $\sigma \in \Sigma$. Hence the sequence
        $(\tau_i - \tau_{i-1})_{i \ge 1}$ is identically distributed
        and for all $\sigma \in \Sigma$,
	$$
	\mathbb E_\sigma \left[ \tau_n - \tau_0 \right] = \sum_{i = 1}^n \mathbb E_\sigma [\tau_i - \tau_{i-1}] = n \sum_{s = -1, 0, 1} \mathbb P_\sigma \left( \sigma(\tau_0) = (s, s) \right) \mathbb E_{(s, s)} \left[ \tau_0 \right] = n  \frac{\alpha^{2} + 2 \, \alpha \beta + \beta^{2}}{2 \, \alpha^{2} \beta + \alpha \beta^{2}}.
	$$
		
	\underline{Step 3.} Define
        $\D_i = \int_{\tau_{i-1}}^{\tau_i} (\sigma^2(s) - \sigma^1(s))
        ds$ for $i \ge 1$ as well as
        $S_n = \sum_{i = 1}^n \D_i = \int_{\tau_0}^{\tau_n}
        (\sigma^2(s) - \sigma^1(s)) ds$ for $n \ge 1$. Start by
        computing
        $\mathbb E_\sigma \left[ e^{\lambda \int_0^{\tau_0}
            (\sigma^2(s) - \sigma^1(s)) ds} \right]$ for all
        $\sigma \in \Delta = \{ (\sigma^1, \sigma^2) \in \Sigma:
        \sigma^1 = \sigma^2\}$.
	
	Let $\hat \tau_0 = \inf \{ t \ge 0 : \sigma^1(t) =
        \sigma^2(t)\}$ be the first hitting time of $\Delta$ and
        notice that $\tau_0$ is the first {\it return} time of
        $\Delta$ so that $\tau_0 = \hat \tau_0$ whenever the initial
        state of $(\sigma^1, \sigma^2)$ is outside
        $\Delta$. Furthermore, if one sets
        $k_\sigma = \mathbb E_\sigma \left[ e^{\lambda \int_0^{\hat
              \tau_0} (\sigma^2(s) - \sigma^1(s)) ds} \right]$ then,
        by first step analysis, the vector
        $(k_\sigma)_{\sigma \in \Sigma}$ satisfies the system of
        linear equations
	\begin{align*}
	\sum_{(s^1, s^2) \in \Sigma} q_{(\sigma^1, \sigma^2), (s^1, s^2)} k_{(s^1, s^2)} + \lambda (\sigma^2 - \sigma^1) k_{(\sigma^1, \sigma^2)} = 0 &\text{ for } (\sigma^1, \sigma^2) \notin \Delta, \\
	k_{(\sigma^1, \sigma^2)} = 1 &\text{ for } (\sigma^1, \sigma^2) \in \Delta
	\end{align*}
	where $(q_{\sigma, \sigma'})_{\sigma, \sigma' \in \Sigma}$ is
        the generator of the couple $(\sigma^1, \sigma^2)$.
	
	Solving this system and using
        $\mathbb E_\sigma \left[ e^{\lambda \int_0^{\tau_0}
            (\sigma^2(s) - \sigma^1(s)) ds} \right] = \sum_{\sigma'
          \neq \sigma} -\frac{q_{\sigma, \sigma'}}{q_{\sigma, \sigma}}
        k_{\sigma'}$ for all $\sigma \in \Delta$ yields
	\begin{align*}
	\mathbb E_\sigma \left[ e^{\lambda \int_0^{\tau_0} (\sigma^2(s) - \sigma^1(s)) ds} \right] = \frac{4 \, \alpha^{4} + 4 \, \alpha^{3} \beta + \alpha^{2} \beta^{2} - 2 \, {\left(2 \, \alpha^{2} + 3 \, \alpha \beta + \beta^{2}\right)} \lambda^{2}}{4 \, \alpha^{4} + 4 \, \alpha^{3} \beta + \alpha^{2} \beta^{2} + 4 \, \lambda^{4} - 4 \, {\left(2 \, \alpha^{2} + 3 \, \alpha \beta + \beta^{2}\right)} \lambda^{2}} =: \phi(\lambda) \text{ for all } \sigma \in \Delta
	\end{align*}
	
	so that for all $\sigma = (\sigma^1_0, \sigma^2_0) \in
        \Sigma$,
	\begin{align*}
          \mathbb{E}_\sigma \left[ e^{\sum_{i = 1}^n \lambda_i \D_i} \right] = \mathbb{E}_\sigma \left[ e^{\sum_{i = 1}^{n - 1} \lambda_i \D_i}  \mathbb{E}_{(\sigma^1(\tau_n), \sigma^2(\tau_n))} \left[ e^{\lambda \int_0^{\tau_0} (\sigma^2(s) - \sigma^1(s)) ds} \right] \right] = \mathbb{E}_\sigma \left[ e^{\sum_{i = 1}^{n - 1} \lambda_i \D_i} \right] \phi(\lambda_n),
	\end{align*}
	which can be iterated to obtain
        $\mathbb{E}_\sigma \left[ e^{\sum_{i = 1}^n \lambda_i \D_i}
        \right] = \prod_{i = 1}^{n} \phi(\lambda_i)$.
	
	The $\D_i$ are thus independent and identically distributed and
        differentiating $\phi$ gives us
	$$
	\mathbb{E}_\sigma \left[ \D_i \right] = \mathbb{E}_\sigma \left[ \D_i^3 \right] = 0, \quad \quad\mathbb{E}_\sigma \left[ \D_i^2 \right] = \frac{4 \, {\left(\alpha + \beta\right)}}{2 \, \alpha^{3} + \alpha^{2} \beta}, \quad \quad \mathbb{E}_\sigma \left[ \D_i^4 \right] = \frac{96 \, {\left(\alpha^{2} + 4 \, \alpha \beta + 2 \, \beta^{2}\right)}}{4 \, \alpha^{6} + 4 \, \alpha^{5} \beta + \alpha^{4} \beta^{2}}.
	$$
	
	\underline{Step 4.} Set
        $n = \left\lceil \alpha^2 \ell^2 \right\rceil$. For all
        $\sigma = (\sigma^1_0, \sigma^2_0) \in K$ one has
	\begin{align*}
	\mathbb{E}_\sigma \left[ S_n^2 \right] = n \mathbb{E}_\sigma \left[ \D_i^2 \right] = \left\lceil \alpha^2 \ell^2 \right\rceil \frac{4 \left( \alpha + \beta \right)}{2 \, \alpha^3 + \alpha^2 \beta} \geq \alpha^2 \ell^2 \frac{4 \left( \alpha + \beta \right)}{2 \, \alpha^3 + \alpha^2 \beta} = \ell^2 \frac{4 \, r + 4}{2 \, r + 1} \geq 2 \ell^2.
	\end{align*}
	
	Hence, by the Paley-Zygmund inequality
	\begin{align*}
	\mathbb{P}_{\sigma} \left( S_n^2 \geq \ell^2 \right) \ge \mathbb{P}_{\sigma} \left ( S_n^2 \geq \frac{1}{2} \mathbb{E}_\sigma \left[ S_n^2 \right] \right) \ge \left( 1 - \frac{1}{2} \right)^2 \frac{\mathbb{E}_\sigma \left[ S_n^2 \right]^2}{\mathbb{E}_\sigma \left[ S_n^4 \right]}.
	\end{align*}
	
	Making use of step 3 one gets
	\begin{align*}
	\frac{\mathbb{E}_\sigma \left[ S_n^2 \right]^2}{\mathbb{E}_\sigma \left[ S_n^4 \right]} &= \frac{n^2 \mathbb{E}_\sigma \left[ \D_i^2 \right]^2}{n \mathbb{E}_\sigma \left[ \D_i^4 \right] + 3 n (n - 1) \mathbb{E}_\sigma \left[ \D_i^2 \right]^2} = \frac{1}{3} \left( 1 + \frac{r^2 + 6 \, r + 3}{n (r + 1)^2} \right)^{-1} \ge \frac{1}{3} \left( 1 + \frac{r^2 + 6 \, r + 3}{(r + 1)^2} \right)^{-1} \ge \frac{1}{12}.
	\end{align*}
	
	Putting it all together, one gets
	\begin{align*}
	\mathbb{P}_\sigma \left( \left| \int_{\tau_0}^{\tau_n} (\sigma^2(s) - \sigma^1(s)) ds \right| \ge \ell \right) \ge \frac{1}{48} \text{ for all } \sigma \in \Sigma.
	\end{align*}
	
	\underline{Step 5.} From step 2, one gets
	\begin{align*}
	\mathbb{E}_\sigma \left[ \tau_n - \tau_0 \right] &= \left\lceil \alpha^2 \ell^2 \right\rceil \frac{\alpha^{2} + 2 \, \alpha \beta + \beta^{2}}{2 \, \alpha^{2} \beta + \alpha \beta^{2}} \le \left(\frac{1}{\alpha} + \frac{1}{\beta} \right) \left( 1 + \alpha^2 \ell^2 \right) \frac{r + 1}{2 \, r + 1} \le \left(\frac{1}{\alpha} + \frac{1}{\beta} \right) \left( 1 + \alpha^2 \ell^2 \right).
	\end{align*}
	
	So, if one sets $t^* = 288 \, \left( 1/\alpha + 1/\beta \right) \left(1 + \alpha^2 \ell^2\right)$ then
	\begin{align*}
	\mathbb{P}_\sigma \left( \sup_{t_1 \le t_2 \le t^*} \left| \int_{t_1}^{t_2} (\sigma^2(s) - \sigma^1(s)) ds \right| \ge \ell \right) &\ge \mathbb{P}_\sigma \left( \left| \int_{\tau_0}^{\tau_n} (\sigma^2(s) - \sigma^1(s)) ds \right| \ge \ell \right) - \mathbb{P}_\sigma \left( \tau_0 \ge 144 \, \left( \frac{1}{\alpha} + \frac{1}{\beta}\right) \right) \\
	&\quad- \mathbb{P}_\sigma \left( \tau_n - \tau_0 \ge 144 \, \left( \frac{1}{\alpha} + \frac{1}{\beta}\right)\left(1 + \alpha^2 \ell^2\right) \right) \\
	&\ge \frac{1}{48} - \frac{1}{144} - \frac{1}{144} = \frac{1}{144}.
	\end{align*}
	using Markov's inequality for the last step.
\end{proof}

The upper bound for the mixing time of the continuous finite tumble process can
now be proven.

\begin{proof}[Upper bound of theorem~\ref{thm:mixing_time} (ii)]
  Let $\epsilon \in (0, 1)$ be given and let $X, \tilde X$ and
  $\tau_1, \tau_2$ be as in definition
  \ref{def:continuous_continuous_coupling}. Set
  $t^* = a \frac{3}{\alpha} + 288 \, k \left(1/\alpha + 1/\beta\right)
  \left(1 + \alpha^2 \ell^2\right)$ for $a > 0$ and $k \in \mathbb N$
  to be fixed later and define
  $\tau_x = \inf \{ t \ge \max\left(\tau_1, \tau_2 \right) : x(t) =
  \tilde x(t) \}$.

  One has, for all $t\ge t^*$, \begin{align*} \sup_\mu \tv{\mu P_t -
      \pi} &\le \sup_ {X_0, \tilde X_0} \mathbb P \left( X(t) \neq
      \tilde X(t) \right) \le \sup_ {X_0, \tilde X_0} \mathbb P \left(
      \tau_x > t^* \right) \\ &\le \sup_ {X_0, \tilde X_0} \Big\{
    \mathbb{P} \left( \tau_x - \max\left (\tau_1, \tau_2 \right) > 288
      k \left(\frac{1}{\alpha} + \frac{1} {\beta}\right) (1 + \alpha^2
      \ell^2) \right) + \mathbb{P} \left( \max \left ( \tau_1, \tau_2
      \right) > a \frac{3}{\alpha}\right) \Big\} \\ &\le
      \left(\frac{143}{144}\right)^k +  \frac{1} {a}, \end{align*} where, in the last
  step, lemma~\ref {lem_crucial_deterministic_lemma} and lemma~\ref
  {lem_variation_lower_bound} were combined and iterated for the first
  term and lemma \ref{lem_mean_speed_coupling_time_slowman_2017} and
  Markov's inequality were used for the second term.
\end{proof}

\section*{Acknowledgments}

All the authors acknowledge the support of the French Agence nationale de la recherche under the grant ANR-20-CE46-0007 (SuSa project). This work has also been (partially) supported by the French Agence nationale de la recherche under the grant ANR-23-CE40-0003 (CONVIVIALITY Project). A.G. has benefited from the support of the Institut Universitaire de France and by a government grant managed by the Agence Nationale de la Recherche under the France 2030 investment plan ANR-23-EXMA-0001. The authors would like to thank Michel Benaïm for Remark~\ref{rem:non_constructive_discrete_continuous_limit}.

\section*{Data availability}
Data sharing not applicable to this article as no datasets were
generated or analysed during the current study.

\section*{Conflict of interest statement}
  
The authors have no competing interests to declare that are relevant
to the content of this article.

\bibliographystyle{alpha}
\bibliography{biblio.bib}
 
\appendix

\section{Invariant probability measure of the continuous process} \label{sec:invariant_probability_of_the_continuous_process}

In \cite{hahn23}, \textcolor{myGreen}{the explicit form of the invariant measure}, as well
as universality classes, were derived for general velocity
processes. For the sake of completeness, the explicit invariant
measure $\pi$ of the continuous process is derived in details for the
particular cases considered here. Note that it was already computed in
\cite{slowman16,slowman17} as the limit of the discrete invariant
measures, an approach made rigorous by proposition
\ref{cor:limit_interchange}.

The key in the search for its invariant measure is the construction of the
continuous process as a PDMP in Section~\ref{secdavis}. This allows us to use
the characterization of the invariant measure in terms of its generator 
$$
\pi \text{ is invariant } \iff \int \mathcal L f d\pi = 0 \text{ for all } f \in D(E).
$$
following from (34.7) and (34.11) of \cite{davis93}.

\subsection{Symmetries of the invariant measure}

The mappings
\begin{align*}
\rho_1 &= \iota^{-1} \circ \left[ (x, \sigma^1, \sigma^2) \mapsto (\ell - x, \sigma^2, \sigma^1)\right] \circ \iota, \\
\rho_2 &= \iota^{-1} \circ \left[ (x, \sigma^1, \sigma^2) \mapsto (\ell - x, -\sigma^1, -\sigma^2)\right] \circ \iota, \\
\rho_3 &= \iota^{-1} \circ \left[ (x, \sigma^1, \sigma^2) \mapsto (x, -\sigma^2, -\sigma^1)\right] \circ \iota
\end{align*}
are such that $f \mapsto f \circ \rho_i$ is a one-to-one mapping of
$D(\mathcal L)$ onto itself and satisfy the relations
$$
\rho_1^2 = \rho^2_2 = \text{id}_E \text{ and } \rho_1 \circ \rho_2 = \rho_2 \circ \rho_1 = \rho_3.
$$

The goal of this section is to show the following proposition, which
proves useful in the search for the explicit form of the invariant
probability.

\begin{prop} \label{prop:symmetries_of_the_invariant_measure} 
	For $i = 1, 2, 3$ one has $\pi = \rho_i \# \pi$.
\end{prop}

The proof of the previous proposition rests on the following lemma,
the proof of which is omitted for the sake of brevity.

\begin{lem}
  For $i = 1, 2, 3$ one has
  $(\mathcal L f) \circ \rho_i = \mathcal L( f \circ \rho_i)$ for all
  $f \in D(\mathcal L)$.
\end{lem}

\begin{proof}[Proposition~\ref{prop:symmetries_of_the_invariant_measure}] Because
  $f \mapsto f \circ \rho_i$ is a one-to-one mapping of
  $D(\mathcal L)$ onto itself, one has that
  $\int \mathcal L f d\pi = 0$ for all $f \in D(\mathcal L)$ implies
$$
0 = \int \mathcal L( f \circ \rho_i) d\pi = \int (\mathcal L f) \circ \rho_i d\pi = \int \mathcal L f d(\rho_i\#\pi) \text{ for all } f \in D(\mathcal{L})
$$
so that $\rho_i\#\pi$ is also an invariant probability measure. The
uniqueness of $\pi$ implies $\pi = \rho_i \# \pi$.
\end{proof}

One has shown that the action of the group
$G = \{\text{id}_E, \rho_1, \rho_2, \rho_3 \} \cong (\mathbb{Z} / 2
\mathbb{Z})^2$ commutes with the generator. Denoting
$p_G : E \rightarrow E / G$ the projection operator, this is a strong
indication that $p_G(X_D(t))$ is also a Markov process which can be
identified as the process in \cite{hahn23}.

\subsection{Invariant measure of the continuous process}


Introduce the following notation for convenience
\begin{align*}
\kappa = \sqrt{(\alpha+\beta)(2\alpha+\beta)/2}, \quad \quad r = \alpha/\beta, \quad \quad \tilde r = \kappa/\beta = \sqrt{(r + 1)(2r + 1)/2}
\end{align*}
as well as
\begin{align*}
\lambda_d &= \frac{1}{r}\left(1 - \frac{\tilde{r}}{ 2\tilde{r} + (2 r + 1)\tanh(\frac{\kappa {\ell}}{2})}\right), \\
\lambda_a &= \frac{\kappa}{4\tilde{r}}\left(1 - \frac{1}{ 2r+ 2 +  2\tilde r\tanh(\frac{\kappa {\ell}}{2})}\right), \\
\lambda_b &= \frac{\kappa}{4 (r+1) \cosh\left(\frac{\kappa \ell}{2}\right)\left(2\tilde r + (2 r+1) \tanh\left(\frac{\kappa \ell}{2}\right)\right)}.
\end{align*}

\newpage

\begin{multicols}{2}

\begin{table}[H]
\centering
\begin{align*}
\begin{array}{|c|c|c|c|}
	\hline
	\sigma & d^0_\sigma & a_\sigma & d^\ell_\sigma \\
	\hline
	(1, 1) & \frac{1}{4(2 + \omega \ell)} & \frac{\omega}{4(2  + \omega\ell)} & \frac{1}{4(2 + \omega \ell)}\\
	(1, -1) & \frac{1}{2(2 + \omega \ell)} & \frac{\omega}{4(2  + \omega\ell)} & 0\\
	(-1, 1) & 0 & \frac{\omega}{4(2  + \omega\ell)} & \frac{1}{2(2 + \omega \ell)}\\
	(-1, -1) & \frac{1}{4(2 + \omega \ell)} & \frac{\omega}{4(2  + \omega\ell)} & \frac{1}{4(2 + \omega \ell)}\\
	\hline
\end{array}
\end{align*}
\caption{Instantaneous tumble}
\label{tab_coeffs_invariant_measure_slowman_2016}
\end{table}

\begin{table}[H]
\centering
\begin{align*}
\begin{array}{|c|c|c|c|c|c|}
\hline
\sigma & d^0_\sigma & d^\ell_\sigma & a_\sigma & b^\text{c}_\sigma & b^\text{s}_\sigma \\
\hline
(1, 1)	& \frac{1}{4r}	& \frac{1}{4r}	& \lambda_a			& \frac{2 r+1}{r}\lambda_b	& 0						\\
(1, 0)	& 1				& 0				& 2r \lambda_a		& (4 r+2)\lambda_b			& 2\tilde r	\lambda_b	\\
(1, -1)	& \lambda_d		& 0				& \lambda_a			& -(2 r+1)\lambda_b			& -2\tilde r\lambda_b	\\
(0, 1)	& 0				& 1				& 2r \lambda_a		& (4r+2)\lambda_b			& -2\tilde r\lambda_b	\\
(0, 0)	& r				& r				& 4r^2\lambda_a		& 4 r (2 r+1)\lambda_b		& 0						\\
(0, -1)	& 1				& 0				& 2r \lambda_a		& (4 r+2)\lambda_b			& 2\tilde r\lambda_b	\\
(-1, 1)	& 0				& \lambda_d		& \lambda_a			& -(2r+1)\lambda_b			& 2\tilde r\lambda_b	\\
(-1, 0)	& 0				& 1				& 2r \lambda_a		& (4r+2)\lambda_b			& -2\tilde r\lambda_b	\\
(-1, -1)& \frac{1}{4r}	& \frac{1}{4r}	& \lambda_a			& \frac{2 r+1}{r}\lambda_b	& 0\\
\hline
\end{array}
	\end{align*}
\caption{Finite tumble}
\label{tab_coeffs_invariant_measure_slowman_2017}
\end{table}

\end{multicols}

\begin{prop} \label{prop_invariant_measure_slowman_2016} The invariant
  probability $\pi$ of the continuous instantaneous tumble process is
  given by
	\begin{align*}
	\pi = \iota^{-1} \# \left(\sum_{\sigma \in \Sigma} \left( d^0_\sigma \delta_0 + a_\sigma dx + d^\ell_\sigma \delta_\ell \right) \otimes \delta_\sigma\right)
	\end{align*}
	with the coefficients summarized in table \ref{tab_coeffs_invariant_measure_slowman_2016}.
\end{prop}

\begin{prop} \label{prop_invariant_measure_slowman_2017}
	The \underline{unnormalized} invariant measure $\pi$ of the continuous finite tumble process is given by
	$$
	\pi = \iota^{-1} \# \left( \sum_{\sigma \in \Sigma} \left(d^0_\sigma \delta_0 + d^\ell_\sigma \delta_\ell + \left(a_\sigma  + b^\text{s}_\sigma \sinh\left(\kappa\left(x - \frac{\ell}{2}\right)\right) + b^\text{c}_\sigma \cosh \left( \kappa \left(x - \frac{\ell}{2}\right)\right)\right) dx \right) \otimes \delta_\sigma\right)
	$$
	with the coefficients summarized in table \ref{tab_coeffs_invariant_measure_slowman_2017}.
\end{prop}

Let first establish
and solve a system ODEs for the density of the invariant measure in
the bulk. First no boundary conditions are taken into account when
solving the system so the solution depends on constants that remain to
be fixed. One then looks at the boundary conditions to determine these
constants. The proof of proposition
\ref{prop_invariant_measure_slowman_2016} is very similar to the proof
of proposition \ref{prop_invariant_measure_slowman_2017} so is not
included here.

\begin{proof}[Proposition~\ref{prop_invariant_measure_slowman_2017}] \underline{System of ODEs in the
   bulk.} For $\sigma \in \Sigma$ let $\pi_\sigma$ be the restriction of
   $\pi$ to $(0, \ell) \times \{ \sigma \}$ seen as measure on $
   (0, \ell)$. For arbitrary but fixed $f_\sigma \in C_c^\infty(
   (0, \ell))$ one has that $f : E_D \rightarrow \mathbb{R}$ defined by $$ f
   (x, \sigma) = f_\sigma(x) \text{ for } (x, \sigma) \in
   (0, \ell)\times \Sigma \text{ and } f(x, \nu) = 0 \text{ otherwise} $$ is
   in the domain $D(\mathcal L)$ of the generator. Rewriting $\int \mathcal L
   f d\pi = 0$ in terms of $f_\sigma$ and $\pi_\sigma$ and using the fact
   that the $f_\sigma$ are arbitrary, one gets $-V F' + \mathcal Q^t F = 0$
   where \begin{align*} F = \begin{pmatrix} \pi_{(1, 1)} \\ \pi_{
   (1, 0)} \\ \pi_{(1, -1)} \\ \pi_{(0, 1)} \\ \pi_{(0, 0)} \\ \pi_{
   (0, -1)} \\ \pi_{(-1, 1)} \\ \pi_{(-1, 0)} \\ \pi_{(-1, -1)} \end
   {pmatrix} \text{, } \mathcal Q = \begin{pmatrix} -2 \alpha  & \alpha  &
   0 & \alpha  & 0 & 0 & 0 & 0 & 0 \\ \frac{\beta }
   {2} & -\alpha -\beta  & \frac{\beta }{2} & 0 & \alpha  & 0 & 0 & 0 & 0 \\
   0 & \alpha  & -2 \alpha  & 0 & 0 & \alpha  & 0 & 0 & 0 \\ \frac{\beta }
   {2} & 0 & 0 & -\alpha -\beta  & \alpha  & 0 & \frac{\beta }{2} & 0 & 0 \\
   0 & \frac{\beta }{2} & 0 & \frac{\beta }{2} & -2 \beta  & \frac{\beta }
   {2} & 0 & \frac{\beta }{2} & 0 \\ 0 & 0 & \frac{\beta }
   {2} & 0 & \alpha & -\alpha -\beta  & 0 & 0 & \frac{\beta }{2} \\ 0 & 0 &
   0 & \alpha  & 0 & 0 & -2 \alpha  & \alpha  & 0 \\ 0 & 0 & 0 & 0 & \alpha &
   0 & \frac{\beta }{2} & -\alpha -\beta  & \frac{\beta }{2} \\ 0 & 0 & 0 &
   0 & 0 & \alpha  & 0 & \alpha  & -2 \alpha \end{pmatrix}, \end{align*} and
   $V = \text{Diagonal}\left(0, -1, -2, 1, 0, -1, 2, 1, 0\right)$. Note that
   differentiation is used in distributional sense here.
	
	This implies
	\begin{align*}
	\pi_{(1, 1)} + \pi_{(-1, -1)} = \frac{1}{4r} \left( \pi_{(1, 0)}  + \pi_{(0, 1)}  + \pi_{(-1, 0)}  + \pi_{(0, -1)}\right), \quad \quad \pi_{(0, 0)} = \frac{r}{2} \left( \pi_{(1, 0)}  + \pi_{(0, 1)}  + \pi_{(-1, 0)}  + \pi_{(0, -1)}\right),
	\end{align*}
	and $G' = A G$ in the distributional sense where
	\begin{align*}
	G= \begin{pmatrix}
	\pi_{(-1, 1)} \\ \pi_{(-1, 0)} + \pi_{(0, 1)} \\ \pi_{(1, 0)} + \pi_{(0, -1)} \\ \pi_{(1, -1)}
	\end{pmatrix} \text{ and } A = \begin{pmatrix}
	-\alpha  & \frac{\beta }{4} & 0 & 0 \\
	2 \alpha  & -\frac{\alpha }{2}-\frac{3 \beta }{4} & \frac{\alpha }{2}+\frac{\beta }{4} & 0 \\
	0 & -\frac{\alpha }{2}-\frac{\beta }{4} & \frac{\alpha }{2}+\frac{3 \beta }{4} & -2 \alpha  \\
	0 & 0 & -\frac{\beta }{4} & \alpha
	\end{pmatrix}.
	\end{align*}
	
	Hence
	$$
	G = 
	c_1 
	\begin{pmatrix}
	-2 r+2 {\tilde r}-1 \\
	8 r-4 {\tilde r}+4 \\
	8 r+4 {\tilde r}+4 \\
	-2 r-2 {\tilde r}-1
	\end{pmatrix} e^{\kappa x} dx
	+ c_2 
	\begin{pmatrix}
	-2 r-2 {\tilde r}-1 \\
	8 r+4 {\tilde r}+4 \\
	8 r-4 {\tilde r}+4 \\
	-2 r+2 {\tilde r}-1
	\end{pmatrix} e^{-\kappa x} dx
	+ c_3
	\begin{pmatrix}
	1 \\
	4 r \\
	4 r \\
	1
	\end{pmatrix} dx
	+ c_4
	\left[
	\begin{pmatrix}
	-{2}/{\alpha} \\
	-{4}/{\beta} \\
	{4}/{\beta} \\
	{2}/{\alpha}
	\end{pmatrix}
	+ x
	\begin{pmatrix}
	1 \\
	4 r \\
	4 r \\
	1
	\end{pmatrix}
	\right]
	dx.
	$$

	The symmetry $\pi = \rho_1 \# \pi$ implies
        $\pi_{(-1, 0)} + \pi_{(0, 1)} = \left(x \mapsto \ell -
          x\right)\#\left(\pi_{(1, 0)} + \pi_{(0, -1)}\right)$ which
        forces $c_2 = e^{\kappa \ell} c_1$ and $c_4 = 0$. Furthermore,
        the symmetry $\pi = \rho_3 \# \pi$ implies
	\begin{align*}
	\pi_{(-1, 0)} = \pi_{(0, 1)} &= \frac{1}{2} \left( \pi_{(-1, 0)} + \pi_{(0, 1)} \right), \\
	\pi_{(1, 0)} = \pi_{(0, -1)} &= \frac{1}{2} \left( \pi_{(1, 0)} + \pi_{(0, -1)} \right), \\
	\pi_{(1, 1)} = \pi_{(-1, -1)} &= \frac{1}{8r} \left( \pi_{(-1, 0)} + \pi_{(0, 1)} + \pi_{(1, 0)} + \pi_{(0, -1)}\right),
	\end{align*}
	so that $F = c_1 E_b^+ e^{\kappa x} dx + c_1 E_b^- e^{\kappa(\ell - x)}dx + c_3 E_a dx$ where
	\begin{align*}
	E_b^+ = 
	\begin{pmatrix}
	\frac{2 r+1}{r} \\
	4 r+2 {\tilde r}+2 \\
	-2 r-2 {\tilde r}-1 \\
	4 r-2 {\tilde r}+2 \\
	4 r (2 r+1) \\
	4 r+2 {\tilde r}+2 \\
	-2 r+2 {\tilde r}-1 \\
	4 r-2 {\tilde r}+2 \\
	\frac{2 r+1}{r}
	\end{pmatrix}, \quad \quad
	E_b^- = 
	\begin{pmatrix}
	\frac{2 r+1}{r} \\
	4 r-2 {\tilde r}+2 \\
	-2 r+2 {\tilde r}-1 \\
	4 r+2 {\tilde r}+2 \\
	4 r (2 r+1) \\
	4 r-2 {\tilde r}+2 \\
	-2 r-2 {\tilde r}-1 \\
	4 r+2 {\tilde r}+2 \\
	\frac{2 r+1}{r}
	\end{pmatrix}, \quad \quad
	E_a =
	\begin{pmatrix}
	1 \\ 2r \\ 1 \\ 2r \\ 4r^2 \\ 2r \\ 1 \\ 2r \\ 1
	\end{pmatrix}.
	\end{align*}
	
	\underline{Boundary conditions.} Because the measure is unique, it is ergodic and hence
	$$
	d^0_{(0, 1)} = \lim_{t \rightarrow +\infty} \frac{1}{t}
        \int_0^t 1_{\{ X(s) = (0, 0, 1) \}} ds = 0
	$$
	since the set $ \left\{s \ge 0 : X(s) = (0, 0, 1) \right\}$ is
        almost surely discrete. Similarly
        $d^0_{(-1, 1)} = d^0_{(-1, 0)} = 0$.

	Let $x_\sigma \in \mathbb R$ be arbitrary but fixed for
        $\sigma \in \Sigma$. It is easy to check that
	$$
	f = \iota^{-1} \circ \left[(x, \sigma) \mapsto \left( x_\sigma \frac{\ell -x}{\ell}, \sigma \right)\right]\circ \iota
	$$
	is in $D(\mathcal L)$. Injecting the formula for $\pi$ into
        $\int \mathcal L f d\pi$ and using the fact that the $x_\sigma$
        are arbitrary, one gets
	\begin{equation}
	\mathcal Q^t D^0 - c_1 V E_b^+ - c_1 e^{\kappa\ell} V E_b^- - c_3 V E_a = 0  \tag{$*$}
	\end{equation}
	where
	\begin{align*}
	D^0 = \left(d^0_\sigma\right)_{\sigma \in \Sigma} = \left(d^0_{(1, 1)}, d^0_{(1, 0)}, d^0_{(1, -1)}, 0, d^0_{(0, 0)}, d^0_{(0, -1)}, 0, 0, d^0_{(-1, -1)}\right)^t.
	\end{align*}
	
	The system of linear equations ($*$) uniquely determines $c_1,
        c_3$ and $D^0$ if one imposes $d^0_{(1, 0)} = 1$. Finally, the
        symmetry $\pi = \rho_1 \# \pi$ implies
        $d^\ell_{(\sigma^1, \sigma^2)} = d^0_{(\sigma^2, \sigma^1)}$
        for all $(\sigma^1, \sigma^2) \in \Sigma$.
\end{proof}

\end{document}